\documentclass{amsart}

\usepackage{ amssymb, xcolor,soul}
\usepackage[all]{xy}
\makeatletter
\@namedef{subjclassname@2020}{\textup{2020} Mathematics Subject Classification}
\makeatother
\usepackage{centernot}

\newenvironment{acknowledgement}{%
	\par\addvspace{12pt}\small\itshape
}{%
	\par\addvspace{12pt}
}

\newtheorem{theorem}{Theorem}[section]
\newtheorem{lemma}[theorem]{Lemma}
\newtheorem{proposition}[theorem]{Proposition}
\newtheorem{corollary}[theorem]{Corollary}

\theoremstyle{definition}
\newtheorem{definition}[theorem]{Definition}
\newtheorem{example}[theorem]{Example}

\newtheorem{problem}[theorem]{Problem}
\theoremstyle{remark}
\newtheorem{remark}[theorem]{Remark}

\newcommand{\dom}{\mbox{\rm dom}}
\newcommand{\range}{\mbox{\rm range}}
\newcommand{\ZFC}{\mathsf{ZFC}}
\newcommand{\diam}{\mbox{\rm diam}}
\newcommand{\full}{\mbox{\rm full}}

\numberwithin{equation}{section}

\begin{document}
\title[an order analysis of hyperfinite Borel equivalence relations]{An order analysis of hyperfinite Borel equivalence relations}


%    author one information
\author{Su Gao}
\address{School of Mathematical Sciences and LPMC, Nankai University, Tianjin 300071, P.R. China}
%\curraddr{}
\email{sgao@nankai.edu.cn}
\thanks{The authors acknowledge the partial support of their research by the National Natural Science Foundation of
China (NSFC) grants 12250710128 and 12271263.
}

%    author two information
\author{Ming Xiao}
\address{School of Mathematical Sciences and LPMC, Nankai University, Tianjin 300071, P.R. China}
%\curraddr{}
\email{ming.xiao@nankai.edu.cn}

%    \subjclass is required.

\date{\today}

%\dedicatory{}
\begin{abstract} In this paper we first consider hyperfinite Borel equivalence relations with a pair of Borel $\mathbb{Z}$-orderings. We define a notion of compatibility between such pairs, and prove a dichotomy theorem which characterizes exactly when a pair of Borel $\mathbb{Z}$-orderings are compatible with each other. We show that, if a pair of Borel $\mathbb{Z}$-orderings are incompatible, then a canonical incompatible pair of Borel $\mathbb{Z}$-orderings of $E_0$ can be Borel embedded into the given pair. We then consider hyperfinite-over-finite equivalence relations, which are countable Borel equivalence relations admitting Borel $\mathbb{Z}^2$-orderings. We show that if a hyperfinite-over-hyperfinite equivalence relation $E$ admits a Borel $\mathbb{Z}^2$-ordering which is self-compatible, then $E$ is hyperfinite.
	 %In this paper we look at the relationship between different Borel class-wise $\mathbb{Z}$-orders generating the same hyperfinite Borel equivalence relation, isolate a property called ``compatible", meaning that they are essentially same. A basic obstruction to compatibility is illustrated in the form of a dichotomy theorem. We then using this notion to proof hyperfiniteness of so called ``hyperfinite-over-hyperfinite" Borel equivalence relations, under some compatibility conditions.
\end{abstract}
%    Abstract is required.

\maketitle

\section{Introduction}

This paper is a contribution to the study of hyperfinite Borel equivalence relations and, more generally, countable Borel equivalence relations which are conjectured to be hyperfinite. Hyperfinite Borel equivalence relations have been studied extensively by many researchers, first in the context of ergodic theory and operator algebras (see e.g. \cite{CFW} and \cite{E1}), and later in the context of descriptive set theory (see e.g. \cite{Weiss}, \cite{SlSt} and \cite{DJK}). Despite an extensive literature, some problems about hyperfinite Borel equivalence relations are stubbornly open. One of the most well-known open problems in this area is Weiss's question (\cite{Weiss}) of whether any orbit equivalence relation induced by a Borel action of a countable amenable group is hyperfinite.

%The main topic of this paper belongs the stream of study of Borel equivalence relations that involves inspecting possible class-wise order types of Borel partial orders that generates them.

By definition, an hyperfinite Borel equivalence relation is an increasing union of a sequence of Borel equivalence relations with finite equivalence classes. In \cite{SlSt}, an equivalent formulation of hyperfiniteness is given: they are precisely those Borel equivalence relations for which there exists a Borel assignment of a linear ordering on each equivalence class so that the order type is a suborder of $\mathbb{Z}$. % (in short, Borel $\mathbb{Z}$-orders).

These equivalent formulations of hyperfiniteness point to two different directions for generalizations. In \cite{DJK}, the class of Borel equivalence relations which are increasing unions of sequences of hyperfinite Borel equivalence relations are introduced. They have been extensively studied since (see also \cite{BJ07}, \cite{JKL2002}). This class of Borel equivalence relations are called {\em hyper-hyperfinite} Borel equivalence relations in some recent literature (e.g. \cite{FriShinVid}). We also adopt this terminology in this work. Whether hyper-hyperfinite Borel equivalence relations are hyperfinite or not is a major open problem in this field, called ``The Union Problem". Another class of Borel equivalence relations, which we call  {\em hyperfinite-over-hyperfinite}, are defined as those Borel equivalence relations allowing an Borel assignment of linear orderings on each equivalence class so that the order type is a suborder of the lexicographic order of $\mathbb{Z}^2$.
%partial order $\prec$ so that its restriction to each $E$-class is order isomorphic to a subset of $\mathbb{Z}^2$ ordered in lexicographical order and two elements are comparable if and only if they are $E$-equivalent.

%It is unknown if these notions are equivalent to hyperfinite, nor if they are equivalent to each other. However, it is known that both of them belong to the class of amenable equivalence relations. Of course, if amenability implies hyperfiniteness, then both questions are immediately answered. However, this question, the Weiss problem (which was originally formulated by using the language of group actions), remains opened for 40 years since asked in \cite{Weiss} and is now one of fundamental problems in this area.

We do not know whether the two notions are equivalent to each other. The question whether any hyper-hyperfinite equivalence relation is hyperfinite is a major open problem in the area, known as the Union Problem (see \cite{DJK}). Kechris (\cite{JKL2002}) defined a notion of Fr{\'e}chet-amenable equivalence relations (this is obviously motivated by and related to Weiss's question) and proved that both hyper-hyperfinite equivalence relations and hyperfinite-over-hyperfinite equivalence relations are Fr{\'e}chet-amenable. We are going to give a more unified characterization of these two properties in Section \ref{prel} and address some related problems there.

What we study in this paper can be regarded as a small part of the more general study on structurable equivalence relations (see the recent paper by Chen and Kechris \cite{CK}). It is notable that Marks proved that any aperiodic countable Borel equivalence relation admits a Borel assignment of a linear ordering on each equivalence class so that the order type is exactly $\mathbb{Q}$ (see \cite[Theorems 1.11 and 8.17]{CK}).

In this paper we study hyperfinite-over-hyperfinite equivalence relations and prove that, under certain conditions, they are hyperfinite. The condition is on the Borel $\mathbb{Z}^2$-ordering of the equivalence relation and is called self-compatible (see Section \ref{hfovhfer} for the definition). Thus our main theorem is stated as follows.

\begin{theorem}\label{thm:main1} If $E$ is a hyperfinite-over-hyperfinite equivalence relation with a Borel $\mathbb{Z}^2$-ordering which is self-compatible, then $E$ is hyperfinite.
\end{theorem}

More generally, the compatibility condition is between two Borel linear orderings on the equivalence classes. Before we prove our main theorem above, we give an analysis of hyperfinite Borel equivalence relations with two Borel $\mathbb{Z}$-orderings and characterize exactly when they are compatible with each other. We show that, if a pair of Borel $\mathbb{Z}$-orderings are incompatible, then there is a Borel embedding of a canonical pair of incompatible Borel $\mathbb{Z}$-orderings of $E_0$ into the given pair. Thus we obtain the following dichotomy theorem.

\begin{theorem}\label{thm:main2} There is a pair $(<_0,<_1)$ of Borel $\mathbb{Z}$-orderings of $E_0$ such that, for any hyperfinite Borel equivalence relation $E$ on a standard Borel space $X$ and a pair $(<, <')$ of Borel $\mathbb{Z}$-orderings of $E$, exactly one of the following holds:
\begin{enumerate}
\item[(I)] $<$ and $<'$ are compatible, or
\item[(II)] There is a Borel embedding $\theta: 2^\omega\to X$ witnessing $E_0\sqsubseteq_B E$ such that $\theta$ is order-preserving from $(<_0,<_1)$ to $(<,<')$.
\end{enumerate}
\end{theorem}

The theorem is proved using Gandy--Harrington forcing. As usual, the technical theorem is an effective version of the main dichotomy theorem in which all objects are $\Delta^1_1$. Some part of our proof is motivated by results of Kanovei \cite{Kanovei1997}, who defined the partial order $<_0$ and considered Borel reductions from $E_0$ to some $E$ which is order-preserving from $<_0$ to some $<$.

%the possible relationships between different Borel $\mathbb{Z}$-orders generating the same equivalence relation. If they are class-wisely same or reverse to each other on a Borel complete section, we say that they are compatible. A pair of incompatbile such partial orders are presented and we are going to show that for any two such orders, either they are compatible or they essentially contain a copy of this example.

%This paper continues the study of Borel equivalence relations from the aspect of class-wise linear orders. This paper is divided into two parts:

%In the second part of this paper, we give an analysis of $\mathbb{Z}^2$-orders by invoking the aforementioned compatibility condition to them, and show that hyperfinite-over-hyperfinite equivalence relations are indeed hyperfinite if they are generated by compatible-with-itself $\mathbb{Z}^2$-orders.

The rest of the paper is organized as follows. In Section 2, we review some basic  concepts and facts. In Section 3 we define the notion of compatible pairs of Borel $\mathbb{Z}$-orderings, state the main dichotomy theorem again, and prove some basic facts. In Section 4 we prove an effective version of the main dichotomy theorem (the technical theorem). In Section 5 we turn to hyperfinite-over-hyperfinite equivalence relations, and formulate and prove Theorem~\ref{thm:main1}.
\begin{acknowledgement}
	Acknowledgement. We would like to express our gratitude to the anonymous referee for the detailed feedback and helpful suggestions.
\end{acknowledgement}

\section{Preliminaries}\label{prel}
The standard notions of descriptive set theory we use in this paper can be found in, e.g., \cite{Moschovakis}, \cite{Kechris1995} and \cite{GaoBook}.

A {\em Polish space} is a separable completely metrizable topological space. If $X$ is a Polish space then the collection of {\em Borel sets} on $X$ is the smallest $\sigma$-algebra of subsets of $X$ containing the open sets. A {\em standard Borel space} is a pair $(X, \mathcal{B})$ where $X$ is a set and $\mathcal{B}$ is a $\sigma$-algebra of subsets of $X$ such that $\mathcal{B}$ is the collection of all Borel sets for some Polish topology on $X$.

Let $X$ be a standard Borel space. An equivalence relation $E$ on $X$ is {\em Borel} if $E$ is a Borel subset of $X^2$. Borel partial orders are similarly defined. Given a subset $A\subseteq X$ and an equivalence relation $E$ on $X$, we denote by $[A]_E=\{x\colon \exists y\in A\,(xEy)\}$ the {\em $E$-saturation} of $A$. $A$ is {\em $E$-invariant} if $A=[A]_E$. When $A=\{x\}$ is a singleton, $[A]_E$ is an $E$-equivalence class ({\em $E$-class} for short), and we write $[x]_E$ for $[A]_E$. A subset $A\subseteq X$ is a {\em complete section} if it has nonempty intersection with every $E$-class. We say that $A$ is an {\em infinite complete section} if it has an infinite intersection with every $E$-class.

An equivalence relation $E$ on a standard Borel space $X$ is {\em finite} (or {\em countable}) if every $E$-class is finite (or countable, respectively). $E$ is {\em hyperfinite} if it is an increasing union of a sequence of finite Borel equivalence relations, i.e., $E=\bigcup_n E_n$, where each $E_n$ is a finite Borel equivalence relation, and $E_n\subseteq E_{n+1}$.

Following \cite{Kechris1991}, we define a {\em Borel structuring} of a countable Borel equivalence relation $E$ on a standard Borel space $X$ as follows. Let $\mathcal{L}=\{R_1,\dots, R_n\}$ be a finite relational language, with $k_i$ being the arity of $R_i$. Let $\mathcal{K}$ be a collection of countable $\mathcal{L}$-structures closed under isomorphism. An assignment $C\mapsto \mathcal{M}_C$, which for each $E$-class $C$ gives an $\mathcal{L}$-structure $\mathcal{M}_C=(C, R_1^C,\dots, R_n^C)$ with universe $C$, is a {\em Borel $\mathcal{K}$-structuring} of $E$ if $\mathcal{M}_C\in\mathcal{K}$ for each $E$-class $C$, and the relations
$$ R_i(x,y_1,\dots,y_{k_i})\iff y_1,\dots, y_{k_i}\in[x]_E\mbox{ and } R_i^{\mathcal{M}_{[x]_E}}(y_1,\dots,y_{k_i}) $$
are Borel subsets of $X^{k_i+1}$. In this paper we only consider two special cases, where $\mathcal{K}$ is either the collection of all suborders of $(\mathbb{Z},<)$ or the collection of all suborders of $(\mathbb{Z}^2, <_{\mbox{\scriptsize lex}})$, where $<_{\mbox{\scriptsize lex}}$ is the lexicographic order of $\mathbb{Z}^2$ (we will write only $\mathbb{Z}^2$ whenever the order we refer to is clear throughout this work). 
We call them {\em Borel $\mathbb{Z}$-orderings} and {\em Borel $\mathbb{Z}^2$-orderings} respectively.

If $\mathcal{K}$ is a collection of countable linear orders and $C\mapsto\mathcal{M}_C$ is a Borel $\mathcal{K}$-structuring of $E$, then one can define a (strict) partial order $<_X$ on $X$ by
$$ x<_Xy\iff x<_{[x]_E}y. $$
Obviously $<_X$ is Borel. This motivates an equivalent but somewhat more intuitive concept as follows.

\begin{definition}
Let $X$ be a set and let $<$ be a (strict) partial order on $X$.
$E_<$ is an equivalence relation defined as $xE_<y$ if and only if $x=y$ or there is a sequence $x=x_0,\dots,x_n=y$ such that $x_i$ and $x_{i+1}$ are comparable in $<$ for each $i<n$. In this case, we say that $<$ generates $E_<$.
\end{definition}

In other words, $E_<$ is the transitive closure of the comparability relation with respect to $<$.

\begin{definition}
Let $L$ be a countable linear order, $X$ be a set, $<$ be a partial order on $X$, and $E=E_{<}$. We say that  $<$ is {\em class-wise $L$} if $<\upharpoonright[x]_E$ is isomorphic to a suborder of $L$ for every $x\in X$.
\end{definition}

The following lemma without proof records the fact that the above two approaches are equivalent.

\begin{lemma} Let $X$ be a standard Borel space and let $E$ be a Borel equivalence relation on $X$. Then the following are equivalent:
\begin{enumerate}
\item[(i)] There is a Borel $\mathbb{Z}$-ordering ($\mathbb{Z}^2$-ordering) of $E$.
\item[(ii)] There is a Borel partial order $<$ on $X$ which is class-wise $\mathbb{Z}$ ($\mathbb{Z}^2$, respectively).
\end{enumerate}
\end{lemma}

In the rest of the paper we work with Borel class-wise $\mathbb{Z}$-orderings and Borel class-wise $\mathbb{Z}^2$-orderings.

\begin{definition} Let $X$ be a Borel equivalence relation and let $E$ be a Borel equivalence relation on $X$.
\begin{enumerate}
\item $E$ is {\em hyper-hyperfinite} if $E$ is the increasing union of a sequence of hyperfinite Borel equivalence relations, i.e., $E=\bigcup_n E_n$, where each $E_n$ is a hyperfinite Borel equivalence relation, and $E_n\subseteq E_{n+1}$.
\item $E$ is {\em hyperfinite-over-hyperfinite} if $E$ admits a Borel $\mathbb{Z}^2$-ordering, or equivalently, there is a Borel class-wise $\mathbb{Z}^2$-ordering on $X$ for $E$.
\end{enumerate}
\end{definition}

The following are some examples of Borel equivalence relations relevant to our study in this paper.
\begin{itemize}
		\item[(a)] $E_0$ is the equivalence relation defined on $2^{\omega}$ by $$xE_0y \iff \exists n<\omega\ \forall m>n\ x(m)=y(m). $$
		\item[(b)] $E_t$ is the equivalence relation defined on $2^{\omega}$ by $$xE_ty \iff \exists n,m<\omega\ \forall k\ x(n+k)=y(m+k).$$
\item[(c)] If $X$ is an uncountable standard Borel space, then $E_t(X)$ is the equivalence relation defined on $X^\omega$ by
    $$ xE_t(X)y\iff \exists n,m<\omega \forall k\ x(n+k)=y(m+k). $$
		\item[(d)] $E_{\mathcal{S}}$ is the equivalence relation defined on $2^{\mathbb{Z}}$ by
$$xE_{\mathcal{S}}y \iff \exists n\in\mathbb{Z}\ \forall k\in\mathbb{Z} \ x(n+k)=y(k). $$

\item[(e)] If $E$ is an equivalence relation over $X$, then $E^{\omega}$ is defined on $X^{\omega}$ by
$$xE^{\omega}y \iff \forall n\ x(n)Ey(n). $$	
	\end{itemize}

The equivalence relations $E_0, E_t, E_{\mathcal{S}}$ are all hyperfinite. Hyperfiniteness can also be characterized in the language of Borel reducibility, as follows.

For Borel equivalence relations $E$ and $F$ on standard Borel spaces $X$ and $Y$ respectively, we say that $E$ is {\em Borel reducible to} $F$, and write $E\leq_B F$, if there is a Borel map $f:X\to Y$ such that
$$ x_1Ex_2\iff f(x_1)Ff(x_2) $$
for all $x_1,x_2\in X$. $f$ is called a {\em Borel reduction} from $E$ to $F$. Moreover, if $f$ can be taken to be injective, then we say that $E$ is {\em Borel embeddable into} $F$, and write $E\sqsubseteq_B F$.

The following are some basic but nontrivial results about hyperfiniteness.
\begin{theorem}\label{thm:hyp} Let $E$ be a Borel equivalence relation on a standard Borel space. Then the following hold:
\begin{enumerate}
	\item[(i)] {\rm (Dougherty--Jackson--Kechris \cite{DJK})} $E$ is hyperfinite if and only if $E\leq_B E_0$.
	\item[(ii)] {\em (Hjorth--Kechris \cite{HK97})} If $E$ is countable, then $E$ is hyperfinite if and only if $E\leq_B E_0^{\omega}$.
\item[(iii)] {\em (Kechris--Louveau \cite{KL}, Dougherty--Jackson--Kechris \cite{DJK})} If $E$ is countable, then for any standard Borel space $X$, $E$ is hyperfinite if and only if $E\leq_B E_t(X)$.
\end{enumerate}
\end{theorem}

We construct more Borel equivalence relations in the following.

\begin{definition}
	Let $E$ be a Borel equivalence relation on a standard Borel space $X$.
\begin{enumerate}

		\item[(f)] $E_0(E)$ is the equivalence relation on $X^{\omega}$ defined by $$E_0(E)=\left\{\,((x_n)_{n<\omega},(y_n)_{n<\omega})\colon \exists N<\omega\ \forall n>N\ (x_nEy_n)\,\right\}. $$
		
		\item[(g)] $E_{\mathcal{S}}(E)$ is the equivalence relation on $X^{\mathbb{Z}}$ defined by $$E_{\mathcal{S}}(E)=\left\{\,((x_n)_{n\in\mathbb{Z}},(y_n)_{n\in\mathbb{Z}})\colon \exists n\in\mathbb{Z}\ \forall k\in\mathbb{Z}\  (x_{n+k}Ey_{k})\,\right\}. $$
	\end{enumerate}
\end{definition}

%It is interested whether these constructions can preserve hyperfiniteness for countable Borel equivalence relations. For example, one of the most fundamental open problems considering hyperfinite equivalence relations can be formulated as:

%A countable Borel equivalence $E$ is called hyperhyperfinite if $E\leq_B E_0(E_0)$, and is called hyperfinite-over-hyperfinite if $E\leq_B E_{\mathcal{S}}(E_0)$. As mentioned in the introduction, it is known that both hyperhyperfinite and hyperfinite-over-hyperfinite Borel equivalence relations are amenable (see, e.g.\cite{JKL2002}). So both problems can be regarded as test questions to the Weiss problem.

%The definitions we give above seem very different from the ones we wrote in the introduction. We will show that they are indeed equivalent to the definitions we gave in the introduction section:

It turns out that hyper-hyperfiniteness and hyperfinite-over-hyperfiniteness can be characterized by Borel reducibility to $E_0(E_0)$ and $E_{\mathcal{S}}(E_0)$, respectively, similarly to the above mentioned result of Hjorth--Kechris. Both of them are not essentially countable (this follows from Proposition \ref{prop:E1} and \ref{prop:E3}) so this characterization is not trivial.

\begin{proposition}\label{hhf}
	A countable Borel equivalence relation $E$ on a standard Borel space $X$ is hyper-hyperfinite if and only if $E\leq_B E_0(E_0)$.
\end{proposition}
\begin{proof}
	Suppose first $E\leq_B E_0(E_0)$ and $E$ is coutable. Without loss of generality, we may assume that $X\subseteq (2^{\omega})^{\omega}$, $E=E_0(E_0)\upharpoonright X$. By the Hjorth--Kechris Theorem (Theorem~\ref{thm:hyp} (ii)), we have
$$ F_n=\{(x,y)\in X^2\colon \forall k>n\ (x_kE_0y_k)\} $$
is hyperfinite on $X$. Then note that $F_n\subseteq F_{n+1}$ and $E=\bigcup F_n$.
	
	Conversely, suppose $E=\bigcup F_n$ is an increasing union of hyperfinite Borel equivalence relations $F_n$. Let $\phi_n$ be a Borel reduction from $F_n$ to $E_0$. Then $(\phi_n(x))_{n<\omega}$ reduces $E$ to $E_0(E_0)$.
\end{proof}

Given a hyperfinite-over-hyperfinite Borel equivalence relation $E$ over $X$ and a $\mathbb{Z}^2$-ordering $<$ witnessing it, it is natural to consider a refinement $F$ defined as $xFy$ if and only if the $<$-interval $[x,y]$ is finite. Then $<$ naturally induces an ordering for $F$-classes in each $E$-class. For further discussion, we need the following concept:

\begin{definition}
	Fix $E$, $<$ and $F$ as above. $x$ is {\em full} if the order type of $F$-classes in $[x]_E$ under $<$ is precisely $\mathbb{Z}$. We call the set of all full points the {\em full part} of $X$ and denote it as $\full(X)$. Its complement is called the {\em non-full part}.  
\end{definition}

Clearly both the full part and the non-full part of $E$ are $E$-invariant. Moreover, the non-full part is comparably trivial:

\begin{proposition}\label{nonfull}
		Given $E$, $<$ and $F$ as above, $E$ restricted to the non-full part is hyperfinite.
\end{proposition}
\begin{proof}
	Let $x\in\full(X)^c$ be in the non-full part. Then the set of $F$-classes in $[x]_E$ has an $<$-extremal element, i.e., either a maximum or a minimum. Let $S=\{y: yFz$ for every $z<y$ or $yFz$ for every $y<z\}$ be the union of such extremal classes. Clearly $S$ is Borel. $<$ becomes a $\mathbb{Z}$-ordering when restricted to $S$, thus making $E\upharpoonright S$ hyperfinite. 
	
	Now $S$ is a Borel complete section of the non-full part, thus $E$ is hyperfinite when restricted to the non-full part.
\end{proof}

\begin{proposition}\label{hoh}
	A countable Borel equivalence relation $E$ on a standard Borel space $X$ is hyperfinite-over-hyperfinite if and only if $E\leq_B E_{\mathcal{S}}(E_0)$.
\end{proposition}
\begin{proof}
	Suppose first $E\leq_B E_{\mathcal{S}}(E_0)$ and $E$ is coutable. We construct a Borel partial order $<$ which generates $E$ and is class-wise $\mathbb{Z}^2$. Without loss of generality, we can assume that $X\subseteq (2^{\omega})^{\mathbb{Z}}$, $E=E_{\mathcal{S}}(E_0)\upharpoonright X$. Consider the Borel $E$-invariant set
$$ A=\{x\in X\colon \exists n\neq 0\ \forall k\in\mathbb{Z}\ (x_{n+k}E_0x_k)\}. $$  Then for each $x\in A$, $[x]_E$ can be characterized by finitely many $E_0$-classes. Hence $E\upharpoonright A$ is hyperfinite, and we can define a Borel partial order $<$ on $A$ which is class-wise $\mathbb{Z}$; in particular it is class-wise $\mathbb{Z}^2$.

It remains to define a Borel class-wise $\mathbb{Z}^2$-ordering on $X\setminus A$ for $E$. Let $F$ be the equivalence relation on $X$ defined by
$$xFy \iff \forall k\in\mathbb{Z}\ (x_kE_0y_k). $$
Clearly $F$ is Borel and $F\subseteq E$, hence $F$ is also countable. By the Hjorth--Kechris Theorem (Theorem~\ref{thm:hyp} (ii)), $F$ is hyperfinite, and we can define Borel partial order $<_*$ on $X$ which is class-wise $\mathbb{Z}$ and generates $F$. Now if $x,y\in X\setminus A$ are from the same $E$-class but different $F$-classes, put
$$x<^*y \iff \exists n>0\ \forall k\in\mathbb{Z}\ (x_{n+k}E_0y_{k}). $$
Clearly $<^*$ is a well-defined Borel partial order on $X\setminus A$ that is $F$-invariant, i.e., if $xFx'$, $yFy'$ and $x<^*y$, then $x'<^*y'$. Also clear is that $<^*$ linearly orders $F$-classes inside a single $E$-class into a $\mathbb{Z}$-order. Now define a Borel partial order $<$ on $X\setminus A$ by
 $$ x<y\iff (xFy \mbox{ and } x<_*y) \mbox{ or } (xEy \mbox{ and } \neg xFy \mbox{ and } x<^*y). $$
 Then $<$ is class-wise $\mathbb{Z}^2$ and generates $E$.

	Conversely, fix a Borel partial order $<$ witnessing that $E$ is hyperfinite-over-hyperfinite. For $xEy$ and $x<y$, define $xFy$ iff $yFx$ iff there is a finite sequence $x=x_0<x_1<\dots<x_k=y$ which is maximal of this form. Then $F$ is a Borel equivalence relation, and each $F$ class is order-embeddable into $\mathbb{Z}$, thus $F$ is hyperfinite. Notice that $<$ orders $F$-classes in a single $E$-class into an order which is order-embeddable into $\mathbb{Z}$. By Proposition \ref{nonfull}, we can focus only on full$(X)$.

	For $x<y$ that are not $F$-equivalent, if for any $z$ such that $x<z<y$ either $xFz$ or $zFy$, we say that $x$ is {\em just below} $y$, or $y$ is {\em just above} $x$; we denote this relation as $B(x,y)$. Clearly $B$ is $F$-invariant. Using countable uniformization on $B$ twice, we obtain a Borel partial injection $\gamma$ such that both $\dom(\gamma)$ and $\range(\gamma)$ are $F$-complete sections, and $B(\gamma(x),x)$ for $x\in \dom(\gamma)$. Now for every $x\in X$ let $p(x)\in \dom(\gamma)\cap [x]_F$ be the closest element to $x$ in the $<$-order (which is a $\mathbb{Z}$-order on $[x]_F$), if a unique such element exists; otherwise there are two such elements, and noting that they are ordered by $<$, so we can let $p(x)$ be the smaller one in the $<$-order. Let $\phi(x)=\gamma(p(x))$. Similarly, define $q(x)\in\range(\gamma)\cap [x]_F$ to be a closest element to $x$ in the $<$-order, and let $\psi(x)=\gamma^{-1}(q(x))$. Then the map $x\mapsto (\dots,\psi^{2}(x),\psi(x),x,\phi(x),\phi^2(x),\dots)$ reduces $E$ to $E_{\mathcal{S}}(F)$. Lastly, $F$ is generated by a $\mathbb{Z}$-ordering, thus hyperfinite. Fixing any reduction witnessing $F\leq_B E_0$, the composition reduces $E$ to $E_{\mathcal{S}}(E_0)$.
\end{proof}	

\begin{remark}
		The function $\gamma$ in above proof is going to be used repeatedly in Section \ref{hfovhfer}.
\end{remark}

As we stated in the introduction, the following problems are open.

\begin{problem}[The Union Problem]
	Is every hyper-hyperfinite equivalence relation hyperfinite?
\end{problem}

%The similar problem for $E_{\mathcal{S}}(E_0)$ is also open:

\begin{problem}[The Hyperfinite-over-Hyperfinite Problem]
	Is every hyperfinite-over-hyperfinite equivalence relation hyperfinite?
\end{problem}

The Union Problem is stated in \cite{DJK} and better known. The Hyperfinite-over-Hyperfinite Problem has been in the folklore, as a special case of Weiss's question. One may reformulate these two problems in the manner of Theorem \ref{thm:hyp}, asking what the exact complexity of those countable Borel equivalence relations reducible to $E_0(E_0)$ or $E_{\mathcal{S}}(E_0)$ is. This leads to a natural question: what are the relationship of $E_0(E_0)$ and $E_{\mathcal{S}}(E_0)$ with $E_1$ and $E_0^{\omega}$? The following propositions give a brief study of their positions in the Borel reducibility hierarchy.

\begin{proposition}\label{prop:E1}
	$E_1\leq_B E_0(E_0).$
\end{proposition}
\begin{proof}
Fix any Borel $\phi$ that reduces $=$ to $E_0$. Then the map $f:(2^{\omega})^{\omega}\to(2^{\omega})^{\omega}$ given by $x\mapsto (\phi(x_n))_{n<\omega}$ reduces $E_1$ to $E_0(E_0)$. 
\end{proof}

\begin{proposition}
	$E_1\nleq_B E_{\mathcal{S}}(E_0)$.	
\end{proposition}
\begin{proof}
	It is known that $E_1$ is not reducible to any orbit equivalence relation (\cite{KL}), so it is enough to show that $E_{\mathcal{S}}(E_0)$ can be obtained as an orbit equivalence relation induced by a Polish group action.
	
	Consider the Polish group $H=\mathbb{Z}^{\mathbb{Z}}$ with the usual coordinate-wise addition, let $\mathbb{Z}$ acts on $H$ naturally via shifting. Then the semiproduct $G=\mathbb{Z}\ltimes H$ is a Polish group. Recall that $E_0$ is an orbit relation obtained by an action of $\mathbb{Z}$ over $2^{\omega}$, so $E_{\mathcal{S}}(E_0)$ is an orbit relation over $(2^{\omega})^{\mathbb{Z}}$ obtained as an action of $G$, as claimed. 
	
\end{proof}

In particular, $E_0(E_0)\nleq_B E_{\mathcal{S}}(E_0)$.

\begin{proposition}\label{prop:E3}
	$E_0^{\omega}\leq_B E_{\mathcal{S}}(E_0).$
\end{proposition}

\begin{proof}

For $x\in (2^{\omega})^{\omega}$, let $$f(x)(n)=\left\{
	\begin{array}{ll}
x(n), &  \mbox{if $n\geq 0$, }\\
1^{(\omega)}, &  \mbox{if $n=-1$,}\\
0^{(\omega)}, &  \mbox{if $n<-1$,}\\
	\end{array}\right.$$
where $0^{(\omega)}$ and $1^{(\omega)}$ denote infinite sequence of only $0$s and $1$s, respectively.  

Then one can easily check that this $f$ is Borel and reduces $E_0^{\omega}$ to $E_0(E_0)$.
\end{proof}

\begin{proposition}
	$E_0^{\omega}\leq_B E_0(E_0)$.
\end{proposition}
\begin{proof}
Fix a sequence $a_0,a_1,a_2,\dots$ of natural numbers so that each natural number appears in the sequence infinitely many times.	Consider the map $f:(2^\omega)^\omega\to(2^\omega)^\omega$ defined by	
	\begin{equation*}
		f(x)=(x(a_0),x(a_1),x(a_2),\dots).
	\end{equation*} 
	This $f$ is clearly Borel and reduces $E_0^{\omega}$ to $E_0(E_0)$.
\end{proof}

To summarize, we have the following diagram of reducibility:

\begin{equation*}
	\xymatrix{
		E_0(E_0) &  & E_{\mathcal{S}}(E_0)\\
		E_1\ar[u] & & E_0^{\omega}\ar[u]\ar[ull]\\
		& E_0\ar[ur]\ar[ul]
	}
\end{equation*}

\section{The main dichotomy theorem}\label{sec:maindich}

We define the notion of compatibility for two Borel class-wise $\mathbb{Z}$-orders for a hyperfinite equivalence relation.

\begin{definition} Let $E$ be a hyperfinite Borel equivalence relation on a standard Borel space $X$. Let $(X,<)$ and $(X,<')$ be two Borel class-wise $\mathbb{Z}$-orders generating $E$. We say that $<$ and $<$ are {\em compatible} if there is a Borel complete section $X'\subseteq X$ such that on each $E$-class restricted to $X'$, either $<=<'$ or $<=>'$. Such an $X'$ is called a {\em $E$-monotonic} subset for $(<, <')$, or just $E$-monotonic, if $<$ and $<'$ are clear from the context. If $<$ and $<'$ are not compatible, we say they are {\em incompatible}.
\end{definition}

Let us first look at an example of an incompatible pair.

\begin{example} Consider the canonical hyperfinite equivalence relation $E_0$ on $2^\omega$. For $x, y\in 2^\omega$, define
$$\begin{array}{rcl}
  x<_0y &\iff& x(n)<y(n) \mbox{ for the largest $n$ such that $x(n)\neq y(n)$,} \\
  x<_1y &\iff& \mbox{for the largest $n$ such that $x(n)\neq y(n)$,} \\
   & &\mbox{$x(n)<y(n)$ if $n$ is odd and $y(n)<x(n)$ if $n$ is even.}
  \end{array}
  $$
  Then $<_0$ and $<_1$ are both Borel class-wise $\mathbb{Z}$-orders for $E_0$, and $<_0$ and $<_1$ are incompatible.
\end{example}

Here is an argument for the incompatibility of the pair $(<_0,<_1)$. As the space $2^{\omega}$ can be covered by countably many homeomorphic images of any complete section, any complete section cannot be meager. Hence it suffices to show that any non-meager subset of $2^\omega$ with the Baire property cannot be $E_0$-monotonic for $(<_0,<_1)$. Let $S$ be a non-meager subset of $2^\omega$ with the Baire property. Let $t\in 2^{<\omega}$ be of even length such that $S$ is comeager in $N_t=\{x\in 2^\omega\colon t\subseteq x\}$. It follows that
$$\begin{array}{l}\{x\in 2^\omega\colon t^\smallfrown 00^\smallfrown x\in S\}, \\
\{x\in 2^\omega\colon t^\smallfrown 10^\smallfrown x\in S\} \mbox{ and} \\
\{x\in 2^\omega\colon t^\smallfrown 11^\smallfrown x\in S\}
\end{array}$$ are all comeager in $2^\omega$ and thus has nonempty intersection. Take an $x$ from their intersection. Then $t^\smallfrown 00^\smallfrown x <_0 t^\smallfrown 10^\smallfrown x$ and $t^\smallfrown 00^\smallfrown x <_1 t^\smallfrown 10^\smallfrown x$, but $t^\smallfrown 10^\smallfrown x <_0 t^\smallfrown 11^\smallfrown x$ while $t^\smallfrown 11^\smallfrown x <_1 t^\smallfrown 10^\smallfrown x$. Thus $<_0$ and $<_1$ agree on some pairs of points in $S$, while disagree on other pairs of points from the same $E_0$-class in $S$. Therefore, $S$ is not $E$-monotonic.

Our main dichotomy theorem states that $(<_0,<_1)$ is a canonical obstruction to compatibility, in the following sense.

\begin{theorem}\label{thm:maindichotomy} Let $E$ be a hyperfinite Borel equivalence relation on a standard Borel space $X$. Let $<$ and $<'$ be Borel class-wise $\mathbb{Z}$-orders generating $E$. Then exactly one of the following holds:
	\begin{enumerate}
		\item[(I)] $<$ and $<'$ are compatible;
		\item[(II)] There is an injective Borel map $\theta\colon 2^{\omega}\to X$ such that $\theta$ reduces $E_0$ to $E$ and $\theta$ is order-preserving from $(<_0,<_1)$ to $(<,<')$, i.e., for any $x, y\in 2^\omega$, $x<_0 y\iff x<y$ and $x<_1y\iff x<'y$.
	\end{enumerate}

\end{theorem}

For a Borel $E$-monotonic subset $A\subseteq X$, we say that $x$ is {\em $E$-class-wise large in $A$} if $[x]_E$ is finite, or $[x]_E$ is infinite and either $<\upharpoonright [x]_E\cap A$ is order-isomorphic to $<\upharpoonright [x]_E$ or $<'\upharpoonright [x]_E\cap A$ is order-isomorphic to $<'\upharpoonright [x]_E$ or both. We say that $A$ is {\em $E$-class-wise large} if every $x\in A$ is $E$-class-wise large in $A$.

The following proposition gives some equivalent characterizations of (I) in the above theorem. It will imply that (I) and (II) in the above theorem are mutually exclusive.

\begin{proposition}\label{prop:completesection} Let $E$ be a hyperfinite Borel equivalence relation on a standard Borel space $X$. Let $<$ and $<'$ be Borel class-wise $\mathbb{Z}$-orders generating $E$. Then the following are equivalent:
	\begin{enumerate}
		\item There is a Borel $E$-monotonic complete section.
        \item There is an $E$-class-wise large Borel $E$-monotonic complete section $A\subseteq X$.
		\item For every Borel complete section $A\subseteq X$, there is a Borel $E$-monotonic complete section $C\subseteq A$.
		\item For every Borel complete section $A\subseteq X$, there is an $E$-class-wise large Borel $E$-monotonic complete section $C\subseteq A$.
	\end{enumerate}
\end{proposition}

\begin{proof}
	Clearly (4)$\Rightarrow$(1). We show (1)$\Rightarrow$(2)$\Rightarrow$(3)$\Rightarrow$(4).

For (1)$\Rightarrow$(2), let $Y\subseteq X$ be a Borel $E$-monotonic complete section. Let $Y'$ be the set of all $y\in Y$ that are not $E$-class-wise large in $Y$.
Then for $y\in Y$, $y\in Y'$ if and only if $[y]_E$ is infinite and exactly one of the following holds:
\begin{enumerate}
\item[(i)] $[y]_E\cap Y$ is finite, and at least one of the following happens:
\begin{itemize}
\item there is either a $<$-least element or a $<$-largest element, but not both, of $[y]_E$;
\item there is either a $<'$-least element or a $<'$-largest element, but not both, of $[y]_E$;
\end{itemize}
\item[(ii)] both $<\upharpoonright [y]_E$ and $<'\upharpoonright [y]_E$ are isomorphic to $\mathbb{Z}$, and there is either a $<$-least element or a $<$-largest element, or both, of $[y]_E\cap Y$.
\end{enumerate}
$Y'$ is Borel. If $Y'=\varnothing$ then there is nothing to prove. Thus we assume $Y'\neq\varnothing$. Let $X'=[Y']_E$. Then $X'$ is a standard Borel space, every $E\upharpoonright X'$-class is infinite, and $Y'$ is a Borel complete section of $E\upharpoonright X'$. To prove (2) it suffices to find a Borel $E\upharpoonright X'$-monotonic complete section $A\subseteq X'$ such that for any $x\in X'$, either $<\upharpoonright [x]_E\cap A$ is order-isomorphic to $<\upharpoonright [x]_E$ or $<'\upharpoonright [x]_E\cap A$ is order-isomorphic to $<\upharpoonright [x]_E$ or both. For notational simplicity, assume $Y'=Y$ and $X'=X$. By picking the least or the largest element as previously mentioned, we obtain a Borel selector $\sigma:X\to X$ for $E$, i.e., a Borel function $\sigma$ such that for all $x\in X$, $\sigma(x)Ex$ and if $x,y\in X$, $\sigma(x)=\sigma(y)$.

For every $x\in X$, define a $2$-coloring
$$ c(x)=\left\{\begin{array}{ll} 0, & \mbox{if either ($x<\sigma(x)$ and $x<'\sigma(x)$) or ($\sigma(x)<x$ and $\sigma(x)<'x$),} \\
1, & \mbox{otherwise.}
\end{array}\right.
$$
Then $c$ is Borel. For any Borel infinite complete section $C\subseteq X$, by the pigeonhole principle, there is a Borel infinite complete section $C'\subseteq C$ such that for any $x\in X$ (as checking whether there are infinitely $x'\in[x]_E$ so that $c(x')=0$ is Borel), $c$ is constant on $[x]_E\cap C'$. $C'$ is an $E$-monotonic set.

Let $X_1$ be the set of all $x\in X$ such that either $<\upharpoonright [x]_E$ or $<'\upharpoonright [x]_E$ is not order-isomorphic to $\mathbb{Z}$. Then $X_1$ is an $E$-invariant Borel subset of $X$. Consider $c\upharpoonright X_1$. There is a Borel infinite complete section $A_1\subseteq X_1$ such that for any $x\in X_1$, $c$ is constant on $[x]_E\cap A_1$. $A_1$ is an $E$-class-wise large $E$-monotonic complete section of $X_1$.

Let $X_2=X\setminus X_1$. Then $x\in X_2$ if and only if both $<\upharpoonright [x]_E$ and $<'\upharpoonright [x]_E$ are order-isomorphic to $\mathbb{Z}$. We note that, for any $x\in X_2$, at least one of the following holds:
\begin{enumerate}
\item[(a)] there are infinitely many $y\in[x]_E$, $y>\sigma(x)$, such that $c(y)=0$, and there are infinitely many $y\in[x]_E$, $y<\sigma(x)$, such that $c(y)=0$;
\item[(b)] there are infinitely many $y\in[x]_E$, $y>\sigma(x)$, such that $c(y)=1$, and there are infinitely many $y\in[x]_E$, $y<\sigma(x)$, such that $c(y)=1$.
\end{enumerate}
Toward a contradiction, assume neither (a) nor (b) holds for some $x\in X_2$. Then there are $x_0, x_1\in [x]_E$, $x_0\leq \sigma(x)\leq x_1$ such that $c(y)$ is constant for all $y>x_1$ and $c(z)$ is constant for all $z<x_0$, but $c(y)\neq c(z)$ for any $z<x_0\leq x_1<y$. For definiteness, assume $c(y)=0$ for $y>x_1$ and $c(z)=1$ for $z<x_0$. Then each of the sets $\{y\in[x]_E\colon y<x_0\}$, $\{y\in[x]_E\colon x_0\leq y\leq x_1\}$, and $\{y\in[x]_E\colon y>x_1\}$ has a $<'$-least element, which implies that $[x]_E$ has a $<'$-least element, a contradiction.

Now assuming one between (a) or (b) holds, we obtain a Borel $E$-monotonic infinite complete section $A_2$ of $X_2$ such that for each $x\in X_2$, both $<\upharpoonright [x]_E\cap A_2$ and $<'\upharpoonright [x]_E\cap A_2$ are order-isomorphic to $\mathbb{Z}$.

Let $A=A_1\cup A_2$. Then $A$ is as required in (2).

Next we prove (2)$\Rightarrow$(3). For this, fix a Borel complete section $Y\subseteq X$ as in (2), as well as another arbitrary Borel complete section $A\subseteq X$. We construct $C$.

Let $X'$ be the set of all $x\in X$ such that the orders $<\upharpoonright [x]_E\cap Y$, $<'\upharpoonright [x]_E\cap Y$, $<\upharpoonright [x]_E\cap A$ and $<'\upharpoonright [x]_E\cap A$ are all order-isomorphic to $\mathbb{Z}$. $X'$ is a Borel $E$-invariant subset of $X$. If $X\neq X'$ then there is a Borel selector $\sigma$ for $X\setminus X'$, and we may assume $\sigma: X\setminus X'\to A$. Then the set $C_0=\{\sigma(x)\colon x\in X\setminus X'\}\subseteq A$ is a Borel $E$-monotonic complete section of $X\setminus X'$. For notational simplicity, we assume $X=X'$ for the rest of the proof.

	For $x,y\in X$ with $xEy$, define
$$d(x,y)=\left\{\begin{array}{ll} 0, &\mbox{ if $x=y$,} \\
|\{z\in Y\colon x<z<y \mbox{ or } y<z<x\}|+1, & \mbox{ otherwise.}
\end{array}\right. $$
Then $d$ is Borel and is a metric on every $[x]_E$.

For each $x\in A$, let $\sigma(x)=\sup_<\{y\in Y\colon y\leq x\}$ and $\sigma'(x)=\sup_{<'}\{y\in Y\colon y\leq' x\}$. Let $h(x)=\min\{d(\sigma(y),\sigma'(y))\colon yEx\}$. Shrink $A$ to
$$A_0=\{x\in A\colon d(\sigma(x),\sigma'(x))=h(x)\},$$ which is clearly still a Borel complete section.

	%Define $R$ on $X'$ to be: \begin{itemize}
		%\item If $<=<'$ on $[x]_E\cap Y$, then $R(x,x')$ if and only if $x<x'$ and $x'<'x$.
		%\item If $<=>'$ on $[x]_E\cap Y$, then $R(x,x')$ if and only if $x<x'$ and $x<'x'$.
	%\end{itemize}
Define a binary relation $R\subseteq E\upharpoonright A$ on $A$ by
$$ R'(x,y)\iff \left\{\begin{array}{ll} x<y\mbox{ and } y<'x, & \mbox{if $<=<'$ on $[x]_E\cap Y$,}\\ x<y\mbox{ and } x<'y, & \mbox{if $<=>'$ on $[x]_E\cap Y$.}
\end{array}\right.
$$
Let $R$ then be the symmetric closure of $R'$. We claim that $R$ is locally finite on $A_0$, i.e., for each $x\in A_0$ there are only finitely many $y\in A_0$ with $R(x,y)$ or $R(y,x)$. To see this, let $(x,y)\in R\upharpoonright A_0$. Then $h(x)=h(y)=d(\sigma(x),\sigma'(x))=d(\sigma(y),\sigma'(y))$. For definiteness, suppose $<=<'$ on $[x]_E\cap Y$ and $R'(x,y)$. Since $x<y$, we have $\sigma(x)\leq \sigma(y)$. Since $y<'x$, we have $\sigma'(y)\leq'\sigma'(x)$. Since $\sigma'(y),\sigma(x)\in [x]_E\cap Y$ and $<=<'$ on $[x]_E\cap Y$, we have $\sigma'(y)\leq \sigma'(x)$. Considering the points $x, \sigma(x), \sigma'(x), \sigma(y), \sigma'(y)$ in the $<$-order, we conclude that
\begin{equation*}
	\begin{array}{ll} 
			 & d(x,\sigma(y))\\
		\leq & d(x,\sigma(x))+d(\sigma(x),\sigma'(x))-d(\sigma'(x),\sigma'(y))+d(\sigma'(y),\sigma(y))\\
		\leq & d(\sigma(x),x)+2h(x). 
	\end{array}
\end{equation*}
Now for a fixed $x$, there are finitely many $z$ such that $d(x,z)\leq d(\sigma(x),x)+2h(x)$, and for each $z\in Y$ there are finitely many $y$ such that $\sigma(y)=z$. Therefore $R$ is locally finite. The cases where $R(y,x)$ holds or $<=>'$ are similar.
	
Now we can obtain a Borel maximal $R$-anticlique $C\subseteq A_0$ (see \cite[Lemma 1.17]{JKL2002}), which is a Borel $E$-monotonic complete section.	

The proof of (3)$\Rightarrow$(4) is identical to the proof of (1)$\Rightarrow$(2).
\end{proof}

With this proposition in mind, we notice that if both (I) and (II) in Theorem~\ref{thm:maindichotomy} hold, then by (1)$\Rightarrow$(3) of the above proposition we can construct a Borel $E$-monotonic complete section in the image of $\theta$ and pull it back to $2^{\omega}$, resulting in a contradiction.

\section{The technical theorem}
Our strategy to prove the main dichotomy theorem is to prove the following effective version of the main dichotomy theorem which we call the Technical Theorem. We state this theorem below in the non-relativized form but from the proof it will be clear that this theorem can be relativized.

\begin{theorem}[The Technical Theorem]\label{effective}
	Let $X$ be a recursively presented Polish space, let $E$ be a $\Delta^1_1$ equivalence relation on $X$ which is generated by $\Delta^1_1$ class-wise $\mathbb{Z}$-orders $<$ and $<'$. Then exactly one of the following holds:
	\begin{enumerate}
		\item[(I)] For every $x\in X$ there is an $\Delta^1_1$ $E$-monotonic subset $S\subseteq X$ so that $x\in S$;
		\item[(II)] There is an injective continuous map that $\theta\colon 2^{\omega}\to X$ such that $\theta$ reduces $E_0$ to $E$ and $\theta$ is order-preserving from $(<_0, <_1)$ to $(<, <')$.
	\end{enumerate}
	
\end{theorem}

We then obtain the following corollary, from which the main dichotomy theorem follows immediately because all Polish spaces of the same cardinality are isomorphic as standard Borel spaces.

\begin{corollary} Let $X$ be a recursively presented Polish space, let $E$ be a hyperfinite Borel equivalence relation on $X$, and let $<$ and $<'$ be Borel class-wise $\mathbb{Z}$-orders on $X$ generating $E$. Then exactly one of following holds:
	\begin{enumerate}
		\item[(I)] $<$ and $<'$ are compatible;
		\item[(II)] There is an injective continuous map $\theta\colon 2^{\omega}\to X$ such that $\theta$ reduces $E_0$ to $E$ and $\theta$ is order-preserving from $(<_0, <_1)$ to $(<, <')$.
	\end{enumerate}
	
\end{corollary}

\begin{proof}
	By relativization, we may assume without loss of generality that $E$, $<$ and $<'$ are $\Delta^1_1$. Note that the second alternate is the same as in Theorem~\ref{effective}. So we only need to show that (I) of Theorem~\ref{effective} implies the first alternate of this corollary.
	
	Suppose (I) of Theorem~\ref{effective} holds. Let $\mathcal{F}=\{S\in\Delta^1_1\colon \mbox{$S\subseteq X$ is $E$-monotonic}\}$. Since there are only countably many $\Delta^1_1$ subsets, we can enumerate the elements of $\mathcal{F}$ as $S_0, S_1,\dots$. For every $x\in X$ there is $n\in\omega$ so that $x\in S_n$.
	
	To construct a Borel $E$-monotonic complete section, we inductively define
$$ \begin{array}{l}
		A_0=S_0, \\
		A_{n+1}=S_n\setminus [A_n]_{E}.
\end{array}
$$
Then each $A_n$ is $E$-monotonic and Borel. Additionally,  $[A_n]_E$ are Borel, $E$-invariant and pairwise disjoint. Therefore $A=\bigcup_{n<\omega} A_n$ is $E$-monotonic and Borel. Since every $x$ is contained in some $S_n$, it must be that $x\in [A]_E$, thus $A$ is a complete section.
\end{proof}

In our treatment of $\Delta^1_1$ hyperfinite equivalence relations, we are going to frequently and tacitly use the fact that quantifiers bounded by $E$-classes ($\forall x\in[y]_E$ and $\exists x\in[y]_E$) are in fact number quantifiers. In the classical setting this is true for any countable Borel equivalence relation, which is a consequence of the Feldman--Moore Theorem (\cite[Theorem 1]{FM}), or in the case of hyperfinite Borel equivalence relations $E$, by direct computations using the equivalent characterization that $E$ is generated by a single Borel automorphism (\cite{SlSt}; also see \cite[Theorem 5.1(4)]{DJK}). In the effective setting, this can be seen by applying e.g. \cite[Theorem 4.5]{Thornton}.

The Gandy-Harrington forcing will be the main tool in our proof of the Technical Theorem. Detailed introductions of Gandy-Harrington forcing can be found in \cite{HKL} and \cite{Kanovei1997}. Here we briefly review some basic notions and prove a few facts to be used in our proof.

For the rest of this section, let $X$ be a fixed recursively represented Polish space. For any natural number $n\geq 1$, the {\em Gandy--Harrington forcing notion} on $X^n$ is the poset
$$ \mathsf{P}_n=\{A\subseteq X^n\colon A\in\Sigma^1_1, \mbox{ $A$ is uncountable}\} $$
ordered by inclusion. The following is a basic fact about the Gandy-Harrington forcing.

\begin{lemma}\label{lem:singleton} Let $n\geq 1$ and let $M$ be a model of sufficiently many axioms of $\ZFC$ with $\mathsf{P}_n\in M$. If $\mathcal{G}$ is $\mathsf{P}_n$-generic over $M$, then $\bigcap \mathcal{G}$ is a singleton $\{x_{\mathcal{G}}\}$ with $x_{\mathcal{G}}\in X^n\cap M[\mathcal{G}]$.
\end{lemma}

Note that $\mathsf{P}_n$ is a different forcing notion from the product $\mathsf{P}_1^n$, but projection maps $\pi$ onto a specific coordinate are open maps for both $\mathsf{P}_n$ and $\mathsf{P}_1^n$.

%Let $T$ be the complete binary tree of height $\omega$, represented as the set of all functions from an initial segment of $\omega$ to $\{0,1\}$ ordered by end extension. In other words, $T=\bigcup_{n<\omega}2^n$.

%In this section (and for following sections), whenever we refer to $2^n$ as a set, we mean the $n$'th level of $T$ (and the root is the $0$'th level). For a branch $b\in T$, denote its initial segment of length $n$ by $b|n=\{b(0),\dots,b(n-1)\} $.

Consider the full binary tree $2^{<\omega}=\bigcup_{n<\omega}2^n$. An element $t\in 2^n$ is a $0,1$-sequence of length $n$. We denote the length of $t$ by $|t|$. If $t=(t_1,\dots, t_n)\in 2^n$ and $m\leq n$, then $t\!\!\upharpoonright\!\! m=(t_1,\dots, t_m)$ denotes the {\em initial segment} of $t$ of length $m$. If $t\in 2^n$, $s\in 2^m$ and $m\leq n$, then we say $t$ {\em extends} $s$, and write $s\subseteq t$ or $t\supseteq s$, if $t\!\upharpoonright\! m=s$. If $t=(t_1,\dots, t_n)\in 2^n$ and $s=(s_1,\dots, s_m)\in 2^m$, then the {\em concatenation} of $t$ and $s$ is $t^\smallfrown s=(t_1,\dots, t_n, s_1,\dots, s_n)\in 2^{n+m}$. When $|s|=1$, instead of writing $t^\smallfrown(i)$ for $i=0,1$, we write $t^\smallfrown i$. Concatenation is an associative operation.

Now for each $n<\omega$, let $\mathsf{P}_{2^n}$ be the Gandy--Harrington forcing on $X^{2^n}$. In this point of view, each $t\in 2^n$ is a coordinate of a point in $X^{2^n}$. We let $\mathsf{P}=\bigcup_{n<\omega}\mathsf{P}_{2^n}$ be the disjoint union of $\mathsf{P}_{2^n}$ (in the same sense as $2^{<\omega}$ being a disjoint union of $2^n$). $\mathsf{P}_{2^n}$ is called the {\em level $n$} of $\mathsf{P}$. For each $p\in \mathsf{P}$, let $\dim(p)$ be the unique $n$ such that $p\in\mathsf{P}_{2^n}$.

%Let $X$ be a recursively representable Polish space. $X^{2^n}$ is the set of all functions from $2^n$ to $X$. Equipped with the usual product topology, it is also a recursively representable Polish space. The $2^n$-dimensional Gandy-Harrington forcing, $P_{2^n}(X)$, is the notion of forcing consists of uncountable $\Sigma^1_1$ subsets of $X^{2^n}$, ordered by inclusion. It is well known (and easy to check) that a generic ultrafilter of $P_{2^n}(X)$ intersects at a singleton in $X^{2^n}$.

%Now, let $P=\bigcup_{n<\omega} P_{2^n}$. For $p\in P$, let $dim(p)$ the unique $n$ that $p\in P_{2^n}$. For each coordinate $t\in 2^n$, we regard $t^\smallfrown 0$ and $t^\smallfrown 1$ coordinates as two copies of $t$. With this in mind,

We define a collection of projections from $\mathsf{P}_{2^n}$ to $\mathsf{P}_{2^m}$ for $n>m$ as follows.

\begin{definition}
	For natural numbers $m\leq n$ and $t\in 2^{n-m}$, the projection map $\pi_{m,t}\colon \mathsf{P}_{2^n}\to \mathsf{P}_{2^m}$, which we call the projection from level $n$ to level $m$ {\em along $t$}, is defined by
$$\pi_{m,t}(p)=\left\{(x_s)_{s\in 2^m}\colon \exists (y_r)_{r\in 2^n}\in p\ \forall s\in 2^m\ y_{s^\smallfrown t}=x_s\right\} $$	
for $p\in\mathsf{P}_{2^n}$.
\end{definition}

The following basic facts are easy to verify. We state them without proof.

\begin{lemma}\label{lem:opendense} Suppose $k\leq m\leq n$, $s\in 2^{m-k}$ and $t\in 2^{n-m}$. Then the following hold: \begin{enumerate}
\item[(i)] $\pi_{k,s^\smallfrown t}=\pi_{k,s}\circ \pi_{m,t}$.
\item[(ii)] If $D\subseteq \mathsf{P}_{2^m}$ is open dense, then so is
$$ \pi_{m,t}^{-1}(D)=\{p\in \mathsf{P}_{2^n}\colon \pi_{m,t}(p)\in D\}. $$
\end{enumerate}
\end{lemma}

Next we define a partial order $\leq_{\mathsf{P}}$ on $\mathsf{P}$ to turn it into a poset.

\begin{definition} Define a partial order $\leq_{\mathsf{P}}$ on $\mathsf{P}$ by letting $p\leq_{\mathsf{P}} q$ iff either $p\subseteq q$ or there are $m<n$ such that $p\in \mathsf{P}_{2^n}$, $q\in \mathsf{P}_{2^m}$ and for every $t\in 2^{n-m}$, $\pi_{m,t}(p)\subseteq q$.
\end{definition}

In this paper, we do not need the full genericity for $\mathsf{P}$. Instead, we use $\mathsf{P}$ to construct objects that are, in the sense of the above projections, simultaneously generic for all $\mathsf{P}_{2^n}$. The sense of sufficient genericity for $\mathsf{P}$ is formulated in the following proposition.

\begin{proposition}\label{prop:genericity} Let $M$ be a countable model of sufficiently many axioms of $\ZFC$ with $\mathsf{P}\in M$. Then there is a sequence of subsets $\{D_n\}_{n<\omega}$ of $\mathsf{P}$ in $M$ such that:
	\begin{itemize}
		\item[(i)] Each $D_n\subseteq \mathsf{P}_{2^n}$ is open dense in $\mathsf{P}_{2^n}$;
		\item[(ii)] If a filter $\mathcal{G}\subseteq \mathsf{P}$ intersects each $D_n$, then for every $b\in 2^{\omega}$ and $n\in\omega$, $$\left\{\pi_{n,b\upharpoonright k}(p)\colon k\in\omega \mbox{ and } \ (p\in \mathsf{P}_{2^{n+k}}\cap\mathcal{G})\right\}$$ is $\mathsf{P}_{2^n}$-generic over $M$.
	\end{itemize}
\end{proposition}
\begin{proof}
	For each $n\in\omega$, enumerate all open dense subsets of $\mathsf{P}_{2^n}$ in $M$ as $\{U^n_k\}_{k<\omega}$. Let $V^n_k=\bigcap_{i\leq k} U^n_i$. Since this is a finite intersection, each $V^n_k\in M$ and is still open dense. Moreover, a filter is $\mathsf{P}_n$-generic over $M$ if and only if it has nonempty intersection with every $V^n_k$ (or, just a tail of $\{V^n_k\}_k$, since it is a decreasing family).
	
	Next, we inductively shrink each $V^n_k$ to an open dense $W^n_k$. Let $W^0_k=V^0_k$ for all $k<\omega$. Suppose we have already defined $W^n_k$ for a fixed $n$ and all $k<\omega$. Define $W^{n+1}_k$ for all $k<\omega$ by induction on $k$:
$$\begin{array}{rcl}
W^{n+1}_0&=&V^{n+1}_0\cap \pi_{n,0}^{-1}(W^n_0)\cap \pi_{n,1}^{-1}(W^n_0), \\
W^{n+1}_k&=&V^{n+1}_k\cap \pi_{n,0}^{-1}(W^n_k)\cap \pi_{n,1}^{-1}(W^n_k)\cap W^{n+1}_{k-1}, \mbox{ for $k>0$}.
\end{array}
$$
By Lemma~\ref{lem:opendense} (ii), all $W^n_k$ are still open dense. By Lemma~\ref{lem:opendense} (i), the two-parameter family $\{W^n_k\}_{n,k}$ satisfies that for every $n>m$ and every $t\in 2^{n-m}$, $\pi_{m,t}^{-1}(W^n_k)\subseteq W^m_k$.
	
Finally, take $D_n=W^n_n$. Then $\{D_n\}_{n<\omega}$ is as required.
\end{proof}

%For any natural number $m$, for any open dense subset $D\subset P_{2^m}$ in $\mathcal{M}$, there is $n\geq m$ so that for every $t\in 2^{n-m}$, $\pi_{m,t}(D_n)\subset D$.

%\section{Proof of the technical theorem}
The rest of this section is devoted to a proof of Theorem~\ref{effective}.

Assume that (I) fails. Then
$$ Y=\left\{x\in X\colon \forall S\in \Delta^1_1\ (\mbox{$S$ is $E$-monotonic} \rightarrow x\not\in S)\right\} $$
is nonempty. $Y$ is $\Sigma^1_1$. Since singletons are $E$-monotonic, $Y$ does not contain any $\Delta^1_1$ member. By the effective Perfect Set Theorem (see, e.g., \cite[4F.1]{Moschovakis}), $Y$ is uncountable. Thus $Y\in \mathsf{P}_1$. We note that $Y$ does not contain any nonempty $\Sigma^1_1$ $E$-monotonic subset of $X$. In fact, by counting quantifiers, we can see that being $E$-monotonic is $\Pi^1_1$ on $\Sigma^1_1$ sets. By the First Reflection Theorem (see, e.g., \cite[Lemma 1.2]{HMS}), every nonempty $\Sigma^1_1$ $E$-monotonic subset of $X$ is included in a nonempty $\Delta^1_1$ $E$-monotonic subset, and is therefore disjoint from $Y$ by definition.

%For each integer $n$, for each uncountable $\Sigma^1_1$ subset $A\subset X^n$, we write $P_n(A)$ for the usual Gandy-Harrington forcing on $X^n$ restricted to $A$. Namely, $P_n(X)=\{S\subset X^n: S\in \Sigma^1_1$ and uncountable $\}$, and $P_n(A)=\{p\leq A:p\in P_n(X)\}$.
%
%For each $k$, we regard $X^{2^{k+1}}$ as $(X^{2^k})\times(X^{2^k})$, and write $\pi_0$ and $\pi_1$
%for projections with respect to these two coordinates. We also need projections to some given coordinates. In this case, we write $\pi_{(l_0,\dots,l_m)}$ for the coordinates labeled with $l_i$. It is a basic property of Gandy-Harrington topology that projections are open maps.
%
%Fix a large enough $H_{\lambda}$ and let $\mathcal{M}$ be a countable elementary submodel of it containing everything we have defined up to last paragraph.
%
%For each $k$, enumerate all open dense subsets $\{D^k_n(X)\}$ of $P_{2^k}(X)$. By taking intersection of initial part of these sequences, we make them decreasing sequences. For uncountable $\Sigma^1_1$ subset $A\subset X^{2^k}$, we also let $D^k_n(A)=\{p\in D^k_n(X):p\leq A\}=D^k_n(X)\cap P_{2^k}(A)$, which clearly decreasing and essentially enumerates open dense subsets of $P_{2^k}(A)$ (Here essentially enumerates open dense subsets means that every filter intersecting them is generic). Now, by deleting enough terms from $D^k_n(X)$, we assume that for each $k$, each $n$ and each $p\in D^{k+1}_n(X)$, $\pi_0(p)\in D^k_n(\pi_0(X))$ and $\pi_1(p)\in D^k_n(\pi_1(X))$.
%

Recall that $<_0$ and $<_1$ are two $\Delta^1_1$ class-wise $\mathbb{Z}$-orders on $2^\omega$ generating $E_0$. Now we extend them to each level of the tree $2^{<\omega}$. For $n<\omega$, $t,s\in 2^n$ and $i=0,1$, define
$$ t<_i s\iff \exists x\in 2^{\omega}\ (t^\smallfrown x<_i s^\smallfrown x). $$
Note that $t<_i s$ if and only if for every $x\in 2^\omega$, $t^\smallfrown x<_i s^\smallfrown x$. Thus $<_i$ is $\Delta^1_1$. The following facts are easy to verify. We state them without proof.

\begin{lemma} For all $n<\omega$ and $t, s\in 2^n$, the following hold:
	\begin{enumerate}
		\item[(i)] For $i=0,1$ and $j=0,1$, $t<_i s$ if and only if $t^\smallfrown j<_i s^\smallfrown j$;
		\item[(ii)] $t^\smallfrown 0<_0 s^\smallfrown 1$;
		\item[(iii)] $t^\smallfrown 0<_1 s^\smallfrown 1$ if $n$ is even, and $t^\smallfrown 1<_1 s^\smallfrown 0$ if $n$ is odd.
	\end{enumerate}
\end{lemma}

Let $M$ be a countable model of sufficiently many axioms of $\ZFC$ with $\mathsf{P}\in M$. Let $\{D_n\}_{n<\omega}$ be as in Proposition~\ref{prop:genericity}. For each $\Sigma^1_1$ set $q\in \mathsf{P}_{2^n}$, write $D_n(q)=\{p\cap q\colon p\in D_n\}$. Clearly these are nonempty sets that are downward closed by the open denseness of $D_n$.

Let $\tau$ be the topology generated by $\mathsf{P}_1$ on $X$. Let $\bar{E}$ be the closure of $E$ in the $\tau\times\tau$ topology (which corresponds to the product forcing $\mathsf{P}_1\times \mathsf{P}_1$). By \cite[Lemma 5.2]{HKL}, $\bar{E}$ is $G_{\delta}$ in $\tau\times \tau$. Also, any partial transversal for $E$ is automatically $E$-monotonic, so $Y$ is included in the set
$$ \{x\in X\colon [x]_{E}\neq [x]_{\bar{E}}\}. $$
By \cite[Lemma 5.3]{HKL}, $E$ is both dense and meager in the relative $\tau\times\tau$ topology on $\bar{E}\cap (Y\times Y)$. Let $\{F_n\}_{n<\omega}$ enumerate the open dense subsets of $\mathsf{P}_1\times\mathsf{P}_1$ restricted to $\bar{E}\cap (Y\times Y)$ in $M$.

Let $<$ and $<'$ be $\Delta^1_1$ class-wise $\mathbb{Z}$-orders on $X$.
%Now, we are going to, on each level $n$, we consider the space $X^{2^n}=X^{2^{n-1}}\times X^{2^{n-1}}$, where these $2^n$ coordinates are labeled by nodes in $2^n$: first $2^{n-1}$ coordinates are those nodes end with $0$, and the later $2^{n-1}$ coordinates are labeled with those nodes end with $1$. We assign to the whole level  $p_n\in P_{2^n}(X)$, that satisfies:
Next we define $p_n\in \mathsf{P}_{2^n}$ with the following properties:
\begin{enumerate}
	\item $p_{n}\in D_n(Y^{2^n})$;
	\item $\pi_{n,i}(p_{n+1})\subseteq p_n$ for $i=0,1$;
	\item For any $x\in p_n$ and $t,s\in 2^n$, $x(t)< x(s)\iff t<_0 s$ and $x(t)<' x(s)$ $\iff  t<_1 s$;
	\item For every pair $t,s$ such that $|t|=|s|$ and $t(|t|)\neq s(|s|)$, $\pi_{0,t}(p_n)\times\pi_{0,s}(p_n)\in F_n$.
\end{enumerate}
To simplify our argument, define
$$u_n=\left\{(x_t)_{t\in 2^n}\in Y^{2^n}\colon  \forall t,s\in 2^n\ (x_t<x_s\iff t<_0 s \mbox{ and } x_t<'x_s\iff t<_1 s)\right\}.$$
Then $u_n\in\mathsf{P}_{2^n}$, and properties (1) and (3) can together be written as $p_n\in D_n(u_n)$.

Granting the existence of such $p_n$, we show that (II) holds. In fact, define $\theta\colon 2^{\omega}\to X$ by
$$ \{\theta(b)\}=\bigcap_{n<\omega} \pi_{0,b\upharpoonright n}(p_n) $$
for $b\in 2^\omega$. To see that this makes sense, note that by properties (1) and (2) and Proposition~\ref{prop:genericity}, the sequence $\pi_{0,b\upharpoonright n}(p_n)$ is $\mathsf{P}_1$-generic over $M$, and therefore, by Lemma~\ref{lem:singleton}, the set on the right hand side above is a singleton. Thus $\theta$ is well defined.

To see that $\theta$ is continuous, let $\rho$ be a compatible metric on $X$. Then for any rational $\epsilon>0$, the set $A=\{p\in \mathsf{P}_1\colon \diam(p)<\epsilon\}$ is open dense in $\mathsf{P}_1$. Thus for any rational $\epsilon>0$, and for any $b\in 2^\omega$, there is $n<\omega$ such that $\pi_{0, b\upharpoonright n}(p_n)\in A$; now if $b'\!\upharpoonright\! n=b\!\upharpoonright\! n$, then $\rho(\theta(b'), \theta(b))<\epsilon$. Thus $\theta$ is continuous.

To see that $\theta$ is injective, note that, by property (3), for any $n<\omega$, $x\in p_n$, and distinct $t, s\in 2^n$, $x(t)\neq x(s)$. It follows that the set $$\left\{p\in \mathsf{P}_{2^n}\colon \forall t,s\in 2^n\ (t\neq s\rightarrow \pi_{0,t}(p)\cap \pi_{0,s}(p)=\varnothing)\right \}$$
is open dense below $p_n\in \mathsf{P}_{2^n}$. Thus if $b, b'\in 2^\omega$ and $b\neq b'$, then there is $n<\omega$ such that $b\!\upharpoonright\! n\neq b'\!\upharpoonright\! n$ and $\pi_{0,b\upharpoonright n}(p_n)\cap \pi_{0, b'\upharpoonright n}(p_n)=\varnothing$; this implies $\theta(b)\neq \theta(b')$. %(since the Gandy-Harrington topology is finer than the Polish topology, which is Hausdorff).

To see that $\theta$ is order-preserving from $(<_0, <_1)$ to $(<, <')$, consider an arbitrary pair $b, b'\in 2^\omega$ with $bE_0b'$ and $b<_0b'$. Let $k\in\omega$ be the largest so that $b(k)\neq b'(k)$. Let $n=k+1$ and $c\in 2^\omega$ be such that $b=(b\!\upharpoonright\!n)^\smallfrown c$. Then $b\!\!\upharpoonright\!\! n<_0 b'\!\!\upharpoonright\!\! n$. By property (3), we have that for any $x\in p_n$, $x(b\!\!\upharpoonright\!\!n)<x(b'\!\!\upharpoonright\!\!n)$. By property (2), we have that for all $m>n$ and $x\in p_m$, $x(b\!\upharpoonright\! m)<x(b'\!\upharpoonright\!m)$. By property (1) and Proposition~\ref{prop:genericity}, the sequence
$$q_m=\pi_{n,c\upharpoonright (m-n)}(p_m), m>n$$
is $\mathsf{P}_{2^n}$-generic over $M$. By Lemma~\ref{lem:singleton}, $\bigcap_{m>n} q_m$ is a singleton, whose only element we denote as $z$. By the definition of $\theta$, we have that $z(b\!\upharpoonright\! n)=\theta(b)$ and $z(b'\!\upharpoonright\! n)=\theta(b')$. By property (3), we have $\theta(b)<\theta(b')$. The proof for $\theta$ being order-preserving from $<_1$ to $<'$ is similar.

Finally, to see that $\theta$ is a reduction from $E_0$ to $E$, note that property (3) implies that for any $b, b'\in 2^\omega$, if $bE_0b'$ then $b$ and $b'$ are in particular $<_0$-comparable, which implies that $\theta(b)$ and $\theta(b')$ are $<$-comparable, thus $\theta(b)E\theta(b')$. By property (4), $(\theta(b),\theta(b'))$ is $\mathsf{P}_1\times \mathsf{P}_1$-generic if $(b,b')\notin E_0$ (c.f. \cite[Section 5]{Kanovei1997}), thus the pair cannot lie in any meager set of the $\mathsf{P}_1\times\mathsf{P}_1$ forcing, in particular $(\theta(b),\theta(b'))\notin E$. We have thus established (II).

Now, let us turn to the construction of $p_n$. For $n=0$ we simply put $p_0=Y$. Inductively, we assume that we have already defined $p_n$ to satisfy properties (1)--(4). By property (3) we have that for any $x\in p_n$ and $t, s\in 2^n$, $x(t)Ex(s)$. We proceed to defining $p_{n+1}$. We assume that $n+1$ is odd. The even case is similar. Our strategy is to construct $p_{n+1}\subseteq u_{n+1}$ to satisfy properties (3) and (4), and then  extend it further to satisfy property (1). Property (2) will be clear from our construction.

We first work with property (4). We enumerate all pairs mentioned in (4) as $\{(t_i,s_i)\}_{i<k}$. Let $p_{n,0,0}=p_{n,1,0}=p_n$. At each step $i<k$, Let $A=\pi_{0,t_i}(p_{n,0,i})$ and $B=\pi_{0,s_i}(p_{n,1,i})$. Note that $(A\times B)\cap \bar{E}$ is nonempty for $i=0$, and we promise that this will be the case for each step $i<k$. Since $A\times B$ is $\tau\times\tau$ open and $F_{n+1}$ is open dense in the relative $\tau\times\tau$ on $\bar{E}$, we are able to find $A'\subseteq A$ and $B'\subseteq B$ so that $A'\times B'\in F_{n+1}$.

Let $$\begin{array}{l}p_{n,0,{i+1}}=\{x\in p_{n,0,i}\colon x(t_i)\in A'\}, \\
p_{n,1,{i+1}}=\{x\in p_{n,1,i}\colon x(s_i)\in B'\}.
\end{array}
$$

Since $E$ is dense in $\bar{E}$, $(A'\times B')\cap E\neq\varnothing$. Let $p_{n,0}=p_{n,0,k}$ and $p_{n,1}=p_{n,1,k}$.
We write $p_{n,0}\otimes p_{n,1}$ for the set
$$\left\{z\in X^{2^{n+1}}\colon \exists x\in p_{n,0}\ \exists y\in p_{n,1}\ \forall t\in 2^n\ (z(t^\smallfrown 0)=x(t)\mbox{ and } z(t^\smallfrown 1)=y(t))\right\}. $$

Clearly, $p_{n,0}\otimes p_{n,1}$, and any of its extensions in $P_{2^{n+1}}$, satisfies property (2) and (4).

Then we turn to (3). We claim that there exists a pair $x\in p_{n,0}$ and $y\in p_{n,1}$ such that both $x(1\dots 1)<y(0\dots 0)$ and $x(1010\dots 10)<'y(0101\dots 01)$. Granting this claim, we set $z(t^\smallfrown 0)=x(t)$ and $z(t^\smallfrown 1)=y(t)$ for all $t\in 2^n$. By the transitivity of $<$ and $<'$ it must be that  $z\in u_{n+1}$ (recall that $11\dots 1$ and $00\dots 0$ are maximal and minimal elements for $<_0$ in their level, and $0101\dots 01$ and $1010\dots 10$ are maximal and minimal elements of $<'_0$ in their level). In particular $(p_{n,0}\otimes p_{n,1})\cap u_{n+1}$ is nonempty. Taking this intersection and extending it to an element of $D_{n+1}$ will give us the required $p_{n+1}$.

%Now, let us show existence of such $(x_{t})_{t\in 2^n}\in p_{n,0}$ and $(x'_{t})_{t\in 2^n}\in p_{n,1}$. Once we have found them we can finish our proof.

Toward a contradiction, assume the claim fails. This means that 
\begin{quote}(*) for any pair $x\in p_{n,0}$ and $y\in p_{n,1}$, we have that $$x(11\dots1)<y(00\dots 0)\iff x(0101\dots01)<'y(1010\dots10).$$
\end{quote}
Since (*) is $\Pi^1_1$ on $\Sigma^1_1$, using the First Reflection Theorem twice, we can extend $p_{n,0}$ and $p_{n,1}$ to $\Delta^1_1$ sets $p'_{n,0}\supseteq p_{n,0}$ and $p'_{n,1}\supseteq p_{n,1}$, respectively, while keeping (*) for $p'_{n,0}$ and $p'_{n,1}$. In addition, we can assure that for $i=0,1$, for any $x\in p_{n,i}$ and $s,t\in 2^n$, $x(s)Ex(t)$. Next we shrink $p'_{n,0}$ and $p'_{n,1}$.

For distinct $x,y\in X$ with $xEy$, define
$$ d(x,y)=|\{z\in X\colon x<z<y \mbox{ or } y<z<x\}|. $$
$d$ is $\Delta^1_1$ and is a metric on each $E$-class. Now, for each $x=(x_t)_{t\in 2^n}\in p_n$ let
$$ \diam(x)=\max\{d(x_t, x_s)\colon t, s\in 2^n\}. $$
This is well defined since $x_tEx_s$ for all $t,s\in 2^n$. Now let
$$ a=\min\{\diam(x)\colon x\in p'_{n,0}\} $$
and define
$$ q_{n,0}=\{x\in p'_{n,0}\colon \diam(x)=a\}. $$
Define a binary relation $R$ on $q_{n,0}$ by
$$ ((x_t)_{t\in 2^n},(y_t)_{t\in 2^n})\in R\iff \exists t,s,r\in 2^n\ (x_{t}\leq y_{s}\leq x_{r} \mbox{ or } y_{t}\leq x_{s}\leq y_{r}). $$
By the definition of $q_{n,0}$, $R$ is locally finite on $q_{n,0}$. $R$ is clearly $\Delta^1_1$, so by \cite[Lemma 1.17]{JKL2002} we are able to find a $\Delta^1_1$ maximal $R$-anticlique $q'_{n,0}\subseteq q_{n,0}$. %In particular, the condition $x_{11\dots1}<x'_{00\dots0}$ actually $E$-class-wisely linearly orders this new $p'_{n,0}$.
We also shrink $p'_{n,1}$ in the same manner to obtain $q'_{n,1}$.

%Let $S=\{x'_{1010\dots10}:(x'_t)_{t\in 2^n}\in p'_{n,1}$, where there is $(x_t)_{t\in 2^n}\in p'_{n,0}$ such that $x_{11\dots1}<x'_{00\dots0}$ and there is no $(x''_t)_{t\in 2^n}\in p'_{n,1}$ satsifying $x_{11\dots1}<x''_{00\dots0}$ while  $x_{0101\dots01}<'x''_{1010\dots10}<'x'_{1010\dots10}\}$. $S$ is still $\Delta^1_1$ as all quantifiers are first order (by Feldman-Moore theorem). $S$ is nonempty as $<'$ is of ordertype $\zeta$ on each $E$-class, and $E$ is nowhere smooth on $X$ (if $E$ is smooth on any $\Sigma^1_1$ subset of $X$ it is easy to construct a $\Sigma^1_1$ $E$-monotonic subset).

Let $q_n$ be the set of all $y\in q'_{n,1}$ such that there is $x\in q'_{n,0}$ with $x(11\dots 1)<y(00\dots0)$ but there is no $z\in q'_{n,1}$ such that $x(11\dots1)<z(00\dots0)$ and
$$x(0101\dots01)<'z(1010\dots10)<'y(1010\dots10). $$
All the quantifiers in the definition of $q_n$ are first-order. Hence $q_n$ is still $\Delta^1_1$. $q_n$ is nonempty since on each $E$-class $<'$ is order-isomorphic to $\mathbb{Z}$.

Now we show that $\pi_{0,1010\dots 10}(q_n)$ is a $E$-monotonic subset of $Y$, which contradicts the definition of $Y$. For this we show that for any $y,z\in q_n$,
$$y(1010\dots10)<z(1010\dots10)\iff y(1010\dots10)<'z(1010\dots10). $$
Assume this fails. Let $y,z\in q_n$ satisfy
$$y(1010\dots 10)<z(1010\dots10) \mbox{ and } z(1010\dots10)<'y(1010\dots10). $$
Since $q_n\subseteq q'_{n,1}$ and $q'_{n,1}$ is an $R$-anticlique, we must have $y(t)<z(s)$ for any $t,s\in 2^n$. Let $x\in q'_{n,0}$ be a witness for $y\in q_n$. Then $$x(11\dots1)<y(00\dots0)<z(00\dots0).$$
By (*) we have $x(0101\dots01)<'z(1010\dots10)$.
Thus $$x(0101\dots01)<'z(1010\dots10)<'y(1010\dots10).$$ This contradicts the definition of $y\in q_n$.

The proof of the technical theorem is thus complete.

\section{Hyperfinite-over-hyperfinite equivalence relations}\label{hfovhfer}

In this final section we consider a special class of hyperfinite-over-hyperfinite Borel equivalence relations and show that they are indeed hyperfinite. To define the class, we consider a hyperfinite-over-hyperfinite equivalence relation and from it define a hyperfinite Borel equivalence relation with two Borel class-wise $\mathbb{Z}$-orders. We then compare the two orders and see if they are compatible. The details are as follows.

Let $E$ be a hyperfinite-over-hyperfinite equivalence relation on a standard Borel space $X$. Suppose $E$ is generated by a Borel class-wise $\mathbb{Z}^2$-ordering $\prec$. Recall from Section \ref{prel} the definition of $\full(X)$. Also recall that the full part is an $E$-invariant Borel subset and $E$ is hyperfinite on the non-full part.

As in the proof of Proposition~\ref{hoh}, we can define a partial Borel injection $\gamma$ on $\full(X)$ so that for any $x\in\dom(\gamma)\subseteq\full(X)$, $[\gamma(x)]_F$ is the immediate predecessor of $[x]_F$ in the $\prec$-order of $F$-classes. Thus $\dom(\gamma)$ is a Borel complete section of $\full(X)$ for $F$. 

For distinct $x,y\in X$ with $xFy$, let $d(x,y)$ be the distance between $x$ and $y$ in the $\prec$-order. If $A\subseteq \full(X)$ is a Borel complete section, then there is a unique Borel function $\eta_A\colon \full(X)\to A$ such that for each $x\in \full(X)$, $d(\eta_A(x),x)=\min\{d(y,x)\colon y\in A\}$, and for any $y\in A$ with $d(y,x)=d(\eta_A(x),x)$, $\eta_A(x)\leq y$. Intuitively, $\eta_A(x)$ is the $d$-closest point to $x$ in $A$, and in the case when $x$ is equidistant to two points of $A$, $\eta_A(x)$ is the $\prec$-smaller one. When $A=\dom(\gamma)$ where $\gamma$ is the partial Borel injection above, $\eta_A$ is finite-to-one.

Now we define two Borel partial orders $<$ and $<'$ on $\full(X)$ so that they both generate $F$. For $x,y\in \full(X)$, let
$$ x<y\iff xFy \mbox{ and } x\prec y $$
and
$$\begin{array}{rcl} x<_{\gamma}'y &\iff& xFy \mbox{ and either } \\
& & (\eta_A(x)=\eta_A(y) \mbox{ and } x\prec y) \mbox{ or } \\
& & [\eta_A(x)\neq\eta_A(y) \mbox{ and } \gamma(\eta_A(x))\prec\gamma(\eta_A(y))],
\end{array} $$
where $A=\dom(\gamma)$. Both $<$ and $<_{\gamma}'$ are Borel class-wise $\mathbb{Z}$-orderings on $\full(X)$ for $F$.

%Note that $\prec$ and $\prec'$ (intersecting $F$) are compatible if and only if there is a $F$-complete section $X'$ on which $\gamma$ is class-wise monotone with respect to $\prec$.

\begin{definition}
	We say that the order $\prec$ is {\em self-compatible} if $<$ and $<_{\gamma}'$ are compatible.
\end{definition}

It may look like the definition depends on the choice of the partial Borel injection $\gamma$. The following proposition shows that this is not the case.

\begin{proposition}\label{prop:selfcompatible} Denote the Borel partial order $<'$ as $<'_{\gamma}$. For any $\gamma_0$ and $\gamma_1$, $<$ and $<'_{\gamma_0}$ are compatible if and only if $<$ and $<'_{\gamma_1}$ are compatible.
\end{proposition}
\begin{proof} Suppose $<$ and $<'_{\gamma_0}$ are compatible. Let $A_0=\dom(\gamma_0)$, $A_1=\dom(\gamma_1)$, $\eta_0=\eta_{A_0}$ and $\eta_1=\eta_{A_1}$. Let $B\subseteq \full(X)$ be a Borel $F$-monotonic complete section for $<$ and $<'_{\gamma_0}$. By Proposition~\ref{prop:completesection} (4), we may assume that $B\subseteq \dom(\gamma_0)$ and that every $x\in\full(X)$ is $E$-class-wise large in $B$.

Let $X_0$ be the set of all $x\in \full(X)$ such that either $<\upharpoonright ([x]_F\cap B)$ or $<'_{\gamma_0}\upharpoonright ([x]_F\cap B)$ is not order-isomorphic to $\mathbb{Z}$. $X_0$ is a Borel $F$-invariant subset of $\full(X)$. On $X_0$ there is a Borel selector $\sigma$, from which we get a Borel $F$-monotonic complete section for $<$ and $<'_{\gamma_1}$. For each $x\in \full(X)\setminus X_0$, both $<\upharpoonright ([x]_F\cap B)$ and $<'_{\gamma_0}\upharpoonright ([x]_F\cap B)$ are order-isomorphic to $\mathbb{Z}$.

Let
$$\begin{array}{l}X_+=\{x\in \full(X)\setminus X_0\colon \forall a, b\in [x]_F\cap B\ (a<b\iff a<'_{\gamma_0}b)\}, \\
X_-=\{x\in \full(X)\setminus X_0\colon \forall a, b\in [x]_F\cap B\ (a<b\iff b<'_{\gamma_0}a)\}.
\end{array}
$$
%For each $x\in Y$, there are two cases: $\gamma$ is increasing on $[x]_F\cap X'$, or is decreasing. Let us divide $X'$ into two parts: $X'_{inc}=\{x:\gamma$ is increasing on $[x]_F\cap X'\}$ and $X'_{dec}=\{x:\gamma$ is decreasing on $[x]_F\cap X'\}$.
Then $X_+$ and $X_-$ are both $F$-invariant Borel sets.

For any $x\in X_+$ and $y\in [\gamma_0(\eta_0(x))]_F$, define
$$\begin{array}{l} I^+_{x,y}=\{z\in [x]_F\cap B\colon x<z \mbox{ and } \gamma_0(\eta_0(z))<y\}, \\
I^-_{x,y}=\{z\in [x]_F\cap B\colon z<x \mbox{ and } y<\gamma_0(\eta_0(z))\}.
\end{array}
$$
Since $x\in X_+$, $\gamma_0$ is increasing on $[x]_F\cap B$, and thus both $I^+_{x,y}$ and $I^-_{x,y}$ are finite, and at least one of them is empty. For $x\in X_+$, let
$$ n(x)=|I^+_{x,\gamma_1(\eta_1(x))}\cup I^-_{x,\gamma_1(\eta_1(x))}|. $$
Let $a=\min\{n(y)\colon yFx\}$ and
  $$X'_+=\{x\in X_+\colon n(x)=a\}. $$
Then $X'_+$ is a Borel complete section of $X_+$ for $F$. Define a binary relation $R$ on $X'_+$ by
$$\begin{array}{rcl} R(x,y)&\iff& (x<y \mbox{ and } \gamma_1(\eta_1(y))<\gamma_1(\eta_1(x))) \mbox{ or } \\
& & (y<x \mbox{ and } \gamma_1(\eta_1(x))<\gamma_1(\eta_1(y))).
\end{array}
$$
$R$ is a Borel graph on $X'_+$. We claim that $R$ is locally finite. To see this, consider $x,y\in X'_+$ with $R(x,y)$. Without loss of generality assume that $x< y$ and $\gamma_1(\eta_1(y))<\gamma_1(\eta_1(x))$. We first show that there are only finitely many $z\in B$ such that $x< z< y$. Consider any such $z\in B$. Note that if $\gamma_0(\eta_0(z))<\gamma_1(\eta_1(y))$, then $\gamma_0(\eta_0(z))<\gamma_1(\eta_1(x))$, and thus $z\in I^+_{x,\gamma_1(\eta_1(x))}$. Similarly if $\gamma_1(\eta_1(x))<\gamma_0(\eta_0(z))$ then $\gamma_1(\eta_1(y))<\gamma_0(\eta_0(z))$ and $z\in I^-_{y,\gamma_1(\eta_1(y))}$.  Therefore there are at most $2a$ such $z$. This in turn implies that there are only finitely many $y$ satisfying our assumption. Thus $R$ is locally finite as claimed.

Let $C_+$ be a Borel maximal $R$-anticlique. Then $C_+$ is Borel $F$-monotonic complete section of $X_+$ for $<$ and $<'_{\gamma_1}$.

A similar construction can be done on $X_-$ to obtain a Borel $F$-monotonic complete section $C_-$ of $X_-$ for $<$ and $<'_{\gamma_1}$.

We have thus shown that $<$ and $<'_{\gamma_1}$ are compatible.
\end{proof}

Now we are ready for the main theorem of this section.

\begin{theorem}\label{thm:selfcomp}
	If $E$ is a hyperfinite-over-hyperfinite equivalence relation on a standard Borel space $X$ and $E$ is generated by a Borel class-wise $\mathbb{Z}^2$-ordering which is self-compatible, then $E$ is hyperfinite.
\end{theorem}

\begin{proof} Because $E$ is hyperfinite on the non-full part of $X$, it suffices to show that $E$ is hyperfinite on $\full(X)$.
	We continue to use the notation developed in the above discussions, in particular the equivalence relation $F$, the partial Borel injection $\gamma$, and the Borel class-wise $\mathbb{Z}$-orderings $<$ and $<_{\gamma}'$ on $\full(X)$ which generate $F$. Let $A=\dom(\gamma)$. Using Proposition \ref{prop:completesection}, let $B\subseteq A$ be a Borel $F$-monotonic complete section of $\full(X)$ such that every $x\in\full(X)$ is $E$-class-wise large in $B$. Now if we consider $\gamma'=\gamma\upharpoonright B$, then $<'_{\gamma'}=<_{\gamma}'$ on $B$ and thus $B$ is still a Borel $F$-monotonic complete section of $\full(X)$ with the stated properties with $<'_{\gamma'}$ replacing $<_{\gamma}'$. Thus we may assume $A=B$, $\gamma=\gamma'$ and omit the subscription and simply write that $<'=<'_{\gamma'}$. Let $\eta=\eta_A=\eta_B$.
	
	%First we look at the case where $\gamma$ is non-decreasing (i.e.$x\preceq y\rightarrow\gamma(x)\preceq\gamma(y)$).
		
	%Note that the non-trivial case is when $\gamma$ is finite to one: when $\gamma$ is monotonic and the same time is constant on an infinite subset of $[x]_F$ for some $x$, its image $\gamma([x]_F)$ either has a lower bound or upper bound in its $F$-class, so both $x$ and $\gamma(x)$ are in the smooth part of $F$, which clearly makes $E$ hyperfinite on these $x$.

	Let $X_0$ be the set of all $x\in \full(X)$ such that either $<\upharpoonright ([x]_F\cap B)$ or $<'\upharpoonright([x]_F\cap B)$ is not order-isomorphic to $\mathbb{Z}$. Let $Y=[X_0]_E$. Then $X_0$ is a Borel $F$-invariant subset of $\full(X)$ and $Y$ is a Borel $E$-invariant subset of $\full(X)$. We claim that $E\upharpoonright Y$ is hyperfinite. To see this, let $\sigma\colon X_0\to X_0$ be a Borel selector for $F$. By the Feldman--Moore Theorem, $\sigma$ can be extended to $Y$, which we still denote as $\sigma$. Then the set $C=\{x\in Y\colon \sigma(x)=x\}$ is a Borel complete section of $E$ on $Y$, and $\prec$ on $C$ is a Borel class-wise $\mathbb{Z}$-ordering. Thus $E\upharpoonright C$ is hyperfinite, and by \cite[Proposition 5.2(4)]{DJK}, $E\upharpoonright Y$ is hyperfinite.

Now let $X_1=\full(X)\setminus Y$. Then for any $x\in X_1$, both $<\upharpoonright([x]_F\cap B)$ and $<'\upharpoonright([x]_F\cap B)$ are order-isomorphic to $\mathbb{Z}$. Define $\phi\colon X_1\to X_1^{\omega}$ by $$\phi(x)=(x,(\gamma\circ \eta)(x),(\gamma\circ \eta)^2(x),\cdots). $$
$\phi$ is clearly Borel. Consider the tail equivalence relation $E_t(X_1)$. We split $X_1$ into two parts:
$$\begin{array}{l}
		X_2=\left\{x\in X_1 \colon \forall y\in[x]_E\ (\phi(x)E_t(X_1)\phi(y))\right\}, \\
		X_3=X_1\setminus X_2.
\end{array}
$$
In other words, $X_2$ is precisely the part of $X_1$ on which $\phi$ reduces $E$ to $E_t(X_1)$. By the theorems of Kechris--Louveau and Dougherty--Jackson--Kechris (Proposition~\ref{thm:hyp} (iii)), $E\upharpoonright X_2$ is hyperfinite.
	
	Now, we focus on $X_3$, which is a Borel $E$-invariant subset of $X_1$. Let $E'=(\phi\times\phi)^{-1}(E_t(X_1))$ be the pullback of $E_t(X_1)$. Then $E'$ is hyperfinite. Note that for any $x,y,z\in X_3$, if $x<y<z$ and $xE'z$ then $xE'yE'z$. Thus the $E'$-classes within an $F$-class consist of intervals in the $<$-order. Let $Z$ be the set of $x\in X_3$ such that there is an infinite $E'$-class within $[x]_F$. Then $Z$ is a Borel $F$-invariant subset of $X_3$, and there is a Borel selector for $F$ on $Z$. A similar argument as the above for $E\upharpoonright Y$ shows that $E\upharpoonright Z$ is hyperfinite. Thus we assume without loss of generality that $Z=\varnothing$, i.e., all $E'$-classes within an $F$-class in $X_3$ are finite.

Let
$$S=\{x\in X_3\colon \forall y<x\ (y,x)\not\in E'\}.$$
Then $S$ is a Borel complete section of $E$ on $X_3$ such that for any $x\in X_3$, $<\upharpoonright ([x]_F\cap S)$ is order-isomorphic to $\mathbb{Z}$. By \cite[Proposition 5.2(4)]{DJK}, we only need to prove that $E\upharpoonright S$ is hyperfinite.

Note that for any $x\in X_3$, $[x]_F\cap [x]_{E'}\cap S$ is a singleton. Therefore, for every $x\in S$ there is a unique $y\in S$ so that $(\gamma\circ\eta)(x)E'y$. Let $\alpha\colon S\to S$ denote this map $x\mapsto y$. Then $\alpha$ is Borel. We note that $\alpha$ is an injection. In fact, if $x,y\in S$ such that $\alpha(x)=\alpha(y)$, then $xFy$ and $xE'y$, hence $x=y$. For every $x\in S$, we define $\alpha^{-\infty}(x)$ to be, if it exists, the unique $z\in S$ so that $z\notin\range(\alpha)$ and there is $n\geq 0$ with $\alpha^n(z)=x$. If such $z$ does not exist we leave $\alpha^{-\infty}(x)$ undefined. $\alpha^{-\infty}$ is a partial Borel function.
	
	Consider the case where $\alpha$ is a bijection on a Borel complete section $S'\subseteq S$ for $E\upharpoonright S$, i.e., $\alpha\colon S'\to S'$ is a bijection from $S'$ onto $S'$. Again, we may assume that for any $x\in S'$, $<\upharpoonright ([x]_F\cap S')$ is order-isomorphic to $\mathbb{Z}$, since otherwise there is a Borel selector for its $F$-class and we deal with such points by a similar argument as that for $E\upharpoonright Y$.  Now we can define a Borel action of the additive group $\mathbb{Z}^2$ on $S'$ by letting $(1,0)\cdot x=\alpha(x)$ and $(0,1)\cdot x=\min_<\{y\in S'\colon y>x\}$. One readily checks that this is indeed a $\mathbb{Z}^2$-action that generates $E\upharpoonright S'$. By a theorem of Weiss (\cite{DJK}), $E\upharpoonright S'$ is hyperfinite, and it follows from \cite[Proposition 5.2(4)]{DJK} that $E\upharpoonright S$ as well as $E\upharpoonright X_3$ is hyperfinite.
	
	So we further focus on $E$-classes in which $\alpha$ is not a bijection on any Borel complete section for $E\upharpoonright S$. Note that for any $x\in S$, if there is $y\in[x]_E$ such that $\alpha^{-\infty}(y)$ is undefined, then we can in a Borel way produce a subset of $[x]_E$ on which $\alpha$ is a bijection. In other words, we consider the part $$S'=\{x\in S\colon \forall y\in[x]_E\ (y\in S\rightarrow y\in \dom(\alpha^{-\infty}))\}.$$
%As usual, we assume without loss of generality that for any $x\in S'$, $<\upharpoonright ([x]_F\cap S')$ is order-isomorphic to $\mathbb{Z}$.
For every $x\in S'$ define
$$\psi(x)=\left\{\begin{array}{ll}\alpha^{-1}(x), & \mbox{ if $x\in \range(\alpha)$,} \\
\alpha^{-1}(\textstyle{\max_{\prec}}\{y\in \range(\alpha)\colon y<x\}), &\mbox{ otherwise}.
\end{array}\right.$$
Then $\psi(x)$ is well defined for every $x\in S'$ and is Borel.
%If for some $x$ the set $\{y\in S:y<_E\sigma^{-\infty}(x)\}$ is empty, then $[x]_{E'}$ is the unique $E'$-class in $[x]_{E|S'}$ having this property and thus $\sigma^{-\infty}(x)$ is a uniquely Borel definable in $[x]_{E_{\gamma}}$. Therefore we can without loss of generality assume that $\phi(x)$ is well-defined.
	
To finish the proof in the case that $\gamma$ is increasing on every $F$-class, we claim that for every pair $x,y\in S'$ with $xEy$, we have that $$(x,\psi(x),...,\psi^n(x),...)E_t(S')(y,\psi(y),...,\psi^m(y),...). $$
To see this, fix such a pair $x,y$. Either there is $n\geq 0$ so that $\alpha^n(x)Fy$ or there is $n\geq 0$ so that $\alpha^n(y)Fx$. Without loss of generality we assume that $\alpha^n(x)Fy$ for some $n\in\mathbb{N}$. Also without loss of generality assume that $\alpha^n(x)<y$. Then for each $k\geq 0$, $\psi^k(\alpha^n(x))$ and $\psi^k(y)$ are in the same $F$-class and $\psi^k(\alpha^n(x))\leq \psi^k(y)$. For each $k\geq 0$, let $N_k$ denote the number of elements of $z\in S'$ so that $\psi^k(\alpha^n(x))<z\leq \psi^k(y)$. When $N_k=0$ we have that $\psi^k(\alpha^n(x))=\psi^k(y)$. Now observe that
\begin{enumerate}
\item[(a)] if $\psi^k(y)\in \range(\alpha)$, then $N_{k+1}\leq N_k$;
\item[(b)] if $\psi^k(y)\notin \range(\alpha)$, then $N_{k+1}< N_k$ if $N_k>0$.
\end{enumerate}
Since $x, y\in S'$ we conclude that case (b) must happen as $k$ increases, and therefore for some large enough $k$, $N_k=0$.
	
Now we extend this result to the general case where $\gamma$ is monotonic on every $F$-class. Consider an equivalence relation $F'$ on $X_3\times\{0,1\}$ defined as $$(x,i)F'(y,j)\iff xFy \mbox{ and } i=j. $$
Define a Borel partial order $\lhd$ on $X_3\times \{0,1\}$ by
$$\begin{array}{rcl} (x,i)\lhd (y,j)&\iff& (x\prec y \mbox{ and } (x,y)\not\in F) \mbox{ or } \\
& & (xFy\mbox{ and } i<j) \mbox{ or }\\
& & (xFy\mbox{ and } x<y\mbox{ and }i=0) \mbox{ or } \\
& & (xFy\mbox{ and }y< x\mbox{ and }i=1).
\end{array}$$
Then $\lhd$ is a Borel class-wise $\mathbb{Z}^2$-ordering on $X_3\times\{0,1\}$. Let $E_\lhd$ be the equivalence relation generated by $\lhd$.

Define $\gamma'$ on $X_3\times\{0,1\}$ by
$$
		\gamma'(x,i)=\left\{\begin{array}{ll}
			(\gamma(x),i), & \mbox{if $\gamma$ is increasing on $[x]_E$,}\\
			(\gamma(x),1-i), & \mbox{otherwise.}
\end{array}\right.
$$	
Let $E_{\gamma'}$ be the equivalence relation generated by $F'$ together with the map $\gamma'$, i.e., $E_{\gamma'}$ is the symmetric and transitive closure of the union of $F'$ and the graph of $\gamma'$. Then $\lhd\cap E_{\gamma'}$ is still a Borel class-wise $\mathbb{Z}^2$-ordering on $X_3\times\{0,1\}$ for $E_{\gamma'}$. Let $\ll$ and $\ll'$ be the induced class-wise $\mathbb{Z}$-ordering on $X_3\times\{0,1\}$ for $F'$. Then $\gamma'$ is increasing with respect to $\ll$ and $\ll'$ on every $F'$-class. By the above argument for the increasing case, $E_{\gamma'}$ is hyperfinite.

Now note that $E_{\gamma'}\subseteq E_\lhd$ and each $E_\lhd$ contains exactly two $E_{\gamma'}$-classes. Hence by \cite[Proposition 1.3(vii)]{DJK}, $E_\lhd$ is hyperfinite. Now $x\mapsto (x,0)$ is a natural Borel embedding of $X_3$ into $X_3\times\{0,1\}$ which is a reduction of $E$ to $E_\lhd$. Thus $E$ is hyperfinite.
\end{proof}

\begin{remark}
	Note that if $E$ is hyperfinite-over-hyperfinite witnessed by a $\mathbb{Z}^2$-ordering $<$, and is indeed hyperfinite, then the $\mathbb{Z}$-ordering $<'$ generating it is trivially self-compatible. Moreover, if we restrict this $<'$ to each $F$-class ($F$ as defined from $<$ in the proof) and order $F$-classes still as in $<$, the resulting ordering is now a $\mathbb{Z}^2$-ordering under which the order of $F$-classes is unchanged and is now self-compatible. So Theorem \ref{thm:selfcomp} actually provides a satisfactory and necessary condition in a strong sense.
\end{remark}

\begin{remark}
	It is still possible that a non-self-compatible $\mathbb{Z}^2$ order generates a hyperfinite Borel equivalence relation. For example, recall $<_0$ and $<_1$ we defined in Section \ref{sec:maindich} are incompatible. Then, let $\gamma(x)(2n)=x(2n+1)$ and $\gamma(x)(2n+1)=1-x(2n+2)$. Define $x<y$ if there is $n>0$ so that $\gamma^n(x)E_0y$ or $x<_0 y$. $<$ is non-self-compatible, since its restriction to each $F$-class is $<_0$ and its pull-back by $\gamma$ is $<_1$. However, we can see that the pull-back of $<$ by $\gamma^2$ is compatible with $<$, so by a similar argument in the last step of the proof of Theorem \ref{thm:selfcomp} the equivalence relation generated by $<$ is still hyperfinite.  
\end{remark}

\bibliographystyle{amsplain}

\end{document}